\documentclass[11pt]{article}% 
\usepackage{amsmath} 
\usepackage{graphicx} 
\usepackage{amsfonts} 
\usepackage{amssymb}% 
\newtheorem{theorem}{Theorem}

\newtheorem{corollary}[theorem]{Corollary}

\newtheorem{example}[theorem]{Example} 
 
\newtheorem{lemma}[theorem]{Lemma}

\newtheorem{proposition}[theorem]{Proposition}

\newenvironment{proof}[1][Proof]{\textbf{#1.} }{\ \rule{0.5em}{0.5em}} 
\begin{document} 
 
\title{Gaussian and non-Gaussian processes of zero power variation, and related 
stochastic calculus} 
\author{\textsc{Francesco Russo} \thanks{ENSTA-ParisTech. Unit\'{e} de 
Math\'{e}matiques appliqu\'{e}es, 828, bd des Mar\'{e}chaux, F-91120 Palaiseau 
(France)} \thanks{INRIA Rocquencourt Projet MathFi and Cermics Ecole des 
Ponts} \textsc{and}\ \textsc{Frederi VIENS }\thanks{Department of Statistics, 
Purdue University, 150 N. University St., West Lafayette, IN 47907-2067, USA }} 
\maketitle 
 
\begin{abstract} 
We consider a class of stochastic processes $X$ defined by $X\left(  t\right) 
=\int_{0}^{T}G\left(  t,s\right)  dM\left(  s\right)  $ for $t\in\lbrack0,T]$, 
where $M$ is a square-integrable continuous martingale and $G$ is a 
deterministic kernel. Let $m$ be an odd integer. Under the assumption that the 
quadratic variation $\left[  M\right]  $ of $M$ is differentiable with 
$\mathbf{E}\left[  \left\vert d\left[  M\right]  (t)/dt\right\vert 
^{m}\right]  $ finite, it is shown that the $m$th power variation% 
\[ 
\lim_{\varepsilon\rightarrow0}\varepsilon^{-1}\int_{0}^{T}ds\left(  X\left( 
s+\varepsilon\right)  -X\left(  s\right)  \right)  ^{m}% 
\] 
exists and is zero when a quantity $\delta^{2}\left(  r\right)  $ related to 
the variance of an increment of $M$ over a small interval of length $r$ 
satisfies $\delta\left(  r\right)  =o\left(  r^{1/(2m)}\right)  $. When $M$ is 
the Wiener process, $X$ is Gaussian; the class then includes fractional 
Brownian motion and other Gaussian processes with or without stationary 
increments. When $X$ is Gaussian and has stationary increments, $\delta$ is 
$X$'s univariate canonical metric, and the condition on $\delta$ is proved to 
be necessary. In the non-stationary Gaussian case, when $m=3$, the symmetric 
(generalized Stratonovich) integral is defined, proved to exist, and its 
It\^{o} formula is established for all functions of class $C^{6}$. 
 
\end{abstract} 
 
%\author{Francesco Russo and Frederi Viens} 

\bigskip 
 
\textbf{MSC Classification 2000}: 60G07; 60G15; 60G48; 60H05. 
 
\section{Introduction} 
 
The purpose of this article is to study wide classes of processes with zero 
cubic variation, and more generally, zero variation of any odd order. Before 
summarizing our results, we give a brief historical description of the topic 
of $p$-variations, as a basis for our motivations. 
 
\subsection{Historical background} 
 
The $p$-variation of a function $f:[0,T]\rightarrow\mathbf{R}$ is the supremum 
over all the possible partitions $\{0=t_{0}<\ldots<t_{N}=T\}$ of $[0,T]$ of 
the quantity $\sum_{i=0}^{N-1}|f(t_{i+1})-f({t_{i}})|^{p}.$ The analytic 
monograph \cite{bruneau} contains an interesting study on this concept, 
showing that a $p$-variation function is the composition of an increasing 
function and a H\"{o}lder-continuous function. The analytic notion of 
$p$-variation precedes stochastic calculus and processes (see \cite{bruneau}). 
 
It was rediscovered in stochastic analysis in the context of pathwise 
stochastic calculus, starting with $p=2$ as in the fundamental paper 
\cite{foellmer} of H. F\"{o}llmer. Dealings with $p$-variations and their 
stochastic applications, particularly to rough path and other recent 
integration techniques for fractional Brownian motion (fBm) and related 
processes, are described at length for instance in the books 
\cite{dudley-norvaisa}, \cite{FV}, and \cite{lyons-qian}, which also contain 
excellent bibliographies on the subject. Prior to this, power variations could 
be seen as related to oscillations of processes in \cite{AW}, and some 
specific cases had been treated, such as local time processes (see 
\cite{walsh}). 
 
The It\^{o} stochastic calculus for semimartingales defines a \emph{quadratic} 
variation of a semimartingale $S$, instead of its $2$-variation, by taking the 
limit in probability of $\sum_{i=0}^{N-1}|S(t_{i+1})-S({t_{i}})|^{2}$ over the 
smaller set of partitions whose mesh tends to $0$, instead of the supremum 
over all partitions. One defines the quadratic variation $[S]$ of $S$ as the 
limit \textit{in probability} of the expression above 
%in (\ref{Q-var}) as 
when the partition mesh goes to $0$, instead of considering pathwise the 
supremum over all partitions, in the hopes of making it more likely to have a 
finite limit; this is indeed the case for standard Brownian motion $M=B$, 
where its 2-variation $\left[  B\right]  $ is a.s. infinite, but its quadratic 
variation is equal to $T$. To reconcile $2$-variations with the finiteness of 
$[B]$, many authors have proposed restricting the supremum over dyadic 
partitions. But there is a fundamental difference between the deterministic 
and stochastic versions of \textquotedblleft variation\textquotedblright, 
since in It\^{o} calculus, quadratic variation is associated with the notion 
of covariation (also known as joint quadratic variation), something which is 
not present in analytic treatments of $2$-variation. The co-variation 
$[S^{1},S^{2}]$ of two semimartingales $S^{1},S^{2}$ is obtained by 
polarization, using again a limit in probability when the partition mesh goes 
to zero. 
 
To work with a general class of processes, the tools of It\^{o} calculus would 
nonetheless restrict the study of covariation to semimartingales. In 
\cite{RV}, the authors enlarged the notion of covariation to general 
processes, in an effort to create a more efficient stochastic calculus tool to 
go beyond semimartingales, by considering regularizations instead of 
discretizations. Drawing some inspiration from the classical fact that a 
continuous $f:[0,T]\rightarrow\mathbf{R}$ has finite variation (1-variation) 
if and only if $\lim_{\varepsilon\rightarrow0}\frac{1}{\varepsilon}\int 
_{0}^{T}|f(s+\varepsilon)-f(s)|ds$ exists, for two processes $X$ and $Y$, 
their covariation $[X,Y]\left(  t\right)  $ is the limit in probability, when 
$\varepsilon$ goes to zero, of 
\begin{equation} 
\left[  X,Y\right]  _{\varepsilon}\left(  t\right)  =\varepsilon^{-1}% 
%TCIMACRO{\tint _{0}^{t}}% 
%BeginExpansion 
{\textstyle\int_{0}^{t}} 
%EndExpansion 
\big(X(s+\varepsilon)-X(s)\big)\big(Y(s+\varepsilon)-Y(s)\big)ds;\quad t\geq0. 
\label{SIVR1Cn}% 
\end{equation} 
$[X,Y]$ coincides with the classical covariation for continuous 
semimartingales. The processes $X$ such that $[X,X]$ exists are called finite 
quadratic variation processes, and were analyzed in \cite{flandoli-russo00, 
RV00}. 
%with several applications, included some to mathematical finance, 
%see \cite{coviello2}. 

The notion of covariation was extended in \cite{RE} to more than two 
processes: the $n$-covariation $[X^{1},X^{2},\cdots,X^{n}]$ of $n$ processes 
$X^{1},\ldots,X^{n}$ is as in formula (\ref{SIVR1Cn}), but with a product of 
$n$ increments, with specific analyses for $n=4$ for fBm with 
\textquotedblleft Hurst\textquotedblright\ parameter $H=1/4$ in \cite{GRV}. If 
$X=X^{1}=X^{2}=X^{3}$ we denote $[X;3]:=[X,X,X]$, which is called the 
\emph{cubic variation}, and is one of the main topics of investigation in our 
article. This variation is the limit in probability of 
\begin{equation} 
\label{X3def}\lbrack X,3]_{\varepsilon}\left(  t\right)  :=\varepsilon 
^{-1}{\textstyle\int_{0}^{t}} \left(  X\left(  s+\varepsilon\right)  -X\left( 
s\right)  \right)  ^{3}ds, 
\end{equation} 
when $\varepsilon\rightarrow0$. \eqref{X3def} involves the signed cubes 
$(X\left(  s+\varepsilon\right)  - X(s))^{3}$, which has the same sign as the 
increment $X\left(  s+\varepsilon\right)  -X(s)$, unlike the case of quadratic 
or $2$-variation, or of the so-called \emph{strong} cubic variation, where 
absolute values are used inside the cube function. Consider the case where $X$ 
is a fBm $B^{H}$ with Hurst parameter $H\in\left(  0,1\right)  $. For fBm, 
\cite{GNRV} establish that $[X,3]\equiv0$ if $H>1/6$ and $[X,3]$ does not 
exist if $H<1/6$, while for $H=1/6$, the regularization approximation $\left[ 
X,3\right]  _{\varepsilon}\left(  t\right)  $ converges in law to a normal law 
for every $t>0$. This phenomenon was confirmed for the related 
finite-difference approximating sequence of $[X,3]\left(  t\right)  $ which 
also converges in law to a Gaussian variable; this was proved in \cite[Theorem 
10]{NOT} by using the the so-called Breuer-Major central limit theorem for 
stationary Gaussian sequences \cite{BM}. 
%where the authors prove that more is 
%true: considered as a process depending on the upper endpoint of the time 
%interval, the approximation converges in law to $\kappa W$ where $W$ is an 
%independent Brownian motion, and $\kappa$ is a universal constant given by% 
%\[ 
%\kappa^{2}=\frac{3}{4}\sum_{r\in\mathbb{Z}}(|r+1|^{\frac{1}{3}}+|r-1|^{\frac 
%{1}{3}}-2|r|^{\frac{1}{3}}). 
%\] 

A practical significance of the cubic variation lies in its well-known ability 
to guarantee the existence of (generalized symmetric) Stratonovich integrals, 
and their associated It\^{o}-Stratonovich formula, 
%with no correction term, 
for various highly irregular processes. This was established in \cite{GNRV} in 
significant generality; technical conditions therein were proved to apply to 
fBm with $H>1/6$, and can extend to similar Gaussian cases with canonical 
metrics that are bounded above and below by multiples of the fBm's, for 
instance the bi-fractional Brownian motion treated in \cite{TR}. A variant on 
\cite{GNRV}'s It\^{o} formula was established previously in \cite{RE} for less 
irregular processes: if $X$ (not necessarily Gaussian) has a finite strong 
cubic variation, so that $[X,3]$ exists (but may not be zero), for $f\in 
C^{3}\left(  \mathbf{R}\right)  $, $f(X_{t})=f(X_{0})+\int_{0}^{t}f^{\prime 
}(X_{s})d^{\circ}X-\frac{1}{12}\int_{0}^{t}f^{\prime\prime\prime}% 
(X_{s})d[X,3]\left(  s\right)  $, which involves the symmetric-Stratonovich 
integral of \cite{RV2}, and a Lebesgue-Stieltjes integral. In \cite{NRS}, an 
analogous formula is obtained for fBm with $H=1/6$, but in the sense of 
distribution laws only: $\int_{0}^{t}f^{\prime}(X_{s})d^{\circ}X$ exist only 
in law, and $\int_{0}^{t}f^{\prime\prime\prime}(X_{s})d[X,3]\left(  s\right) 
$ is replaced by a conditionally Wiener integral defined in law by replacing 
$[X,3]$ with a term $\kappa W$, where $W$ is the independent Wiener process 
identified in \cite{NOT}. 
 
\subsection{Specific motivations} 
 
%Let $[0,T]$ be a fixed time interval. 
Our work herein is motivated by the properties described in the previous 
paragraph, particularly as in \cite{GNRV}. We want to avoid situations where 
It\^{o} formulas can only be established in law, i.e. involving conditionally 
Wiener integrals defined as limits in a weak sense. Thus we study scales where 
this term vanishes in a strong sense, while staying as close to the threshold 
$H=1/6$ as possible. Other types of stochastic integrals for fBm and related 
irregular Gaussian processes make use of the Skorohod integral, identified as 
a divergence operator on Wiener space (see \cite{Nbook} and also \cite{alos, 
biagini, carmona, MV, kruk}), and rough path theory (see \cite{FV, 
lyons-qian}). The former method is not restrictive in how small $H$ can be 
(see \cite{MV}), but is known not to represent a pathwise notion of integral; 
the latter is based in a true pathwise strategy 
%is true pathwise integration, but is largely restricted to $H>1/4$, 
and it is based on giving a L\'{e}vy-type area or iterated integrals \emph{a 
priori}. In principal the objective of the rough path approach is not to link 
any discretization (or other approximation) scheme. These provide additional 
motivations for studying the regularlization methodology of \cite{RV} or 
\cite{RV2}, which does not feature these drawbacks for $H>1/6$. 
 
%Specifically, following \cite{RV, RV2}, the cubic variation $[X,3]\left( 
%t\right)  $ of a process $X$ was defined in \cite{RE} as the limit in 
%probability or in the mean square, as $\varepsilon\rightarrow0$, of 
We come back to the cubic variation approximation $[X,3]$ defined via the 
limit of \eqref{X3def}. 
%as alluded to above. 
The reasons for which $\left[  X,3\right]  =0$ for fBm with $H>1/6$, which is 
considerably less regular than the threshold $H>1/3$ one has for 
$H$-H\"{o}lder-continuous deterministic functions, are the odd symmetry of the 
cube function, and the accompanying probabilistic symmetries of the process 
$X$ itself (e.g. Gaussian property). This doubling improvement over the 
deterministic case does not typically hold for non-symmetric variations: $H$ 
needs to be larger to guarantee existence of the variation; for instance, when 
$X$ is fBm, its strong cubic variation, defined as the limit in probability of 
$\varepsilon^{-1}\int_{0}^{t}\left\vert X\left(  s+\varepsilon\right) 
-X\left(  s\right)  \right\vert ^{3}ds$, exists for $H\geq1/3$ only. 
 
Finally, some brief notes in the case where $X$ is fBm with $H=1/6$. This 
threshold is a critical value since, as mentioned above, whether in the sense 
of regularization or of finite-difference, the approximating sequences of 
$[X,3]\left(  t\right)  $ converge in law to Gaussian laws. In contrast to 
these normal convergences, in our article, we show as a preliminary result 
(Proposition \ref{ChaosConv} herein), that $[X,3]_{\varepsilon}$ does not 
converge in probability for $H=1/6$; the non-convergence of 
$[X,3]_{\varepsilon}$ in probability for $H<1/6$ was known previously. 
 
\subsection{Summary of results and structure of article} 
 
This article investigates the properties of cubic and other odd power 
variations for processes which may not be self-similar, or have stationary 
increments, or be Gaussian, when they have $\alpha$-H\"{o}lder-continuous 
paths; this helps answer to what extent the threshold $\alpha>1/6$ is sharp 
for $[X,3]=0$. We consider processes $X$ defined on $[0,T]$ by a Volterra 
representation% 
\begin{equation} 
X\left(  t\right)  =\int_{0}^{T}G\left(  t,s\right)  dM\left(  s\right)  , 
\label{defX}% 
\end{equation} 
where $M$ is a square-integrable martingale on $[0,T]$, and $G$ is a 
non-random measurable function on $[0,T]^{2}$, which is square-integrable in 
$s$ with respect to $d\left[  M\right]  _{s}$ for every fixed $t$. The 
quadratic variations of these martingale-based convolutions was studied in 
\cite{errami-russoCRAS}. The \textquotedblleft Gaussian\textquotedblright% 
\ case results when $M$ is the standard Wiener process (Brownian motion) $W$. 
 
In this paper, we concentrate on processes $X$ which are not more regular than 
standard Brownian motion; this irregularity is expressed via a concavity 
condition on the squared canonical metric $\delta^{2}\left(  s,t\right) 
=\mathbf{E}\left[  \left(  X\left(  t\right)  -X\left(  s\right)  ^{2}\right) 
\right]  $. It is not a restriction since the main interest of our results 
occurs around the H\"{o}lder exponent $1/(2m)$ for odd $m\geq3$, and processes 
which are more regular than Brownian motion can be treated using classical 
non-probabilistic tools such as the Young integral. 
 
After providing some definitions [Section \ref{DEF}], our first main finding 
is that the processes with zero odd $m$th variation (same definition as for 
$[X,3]=0$ in (\ref{X3def}) but with $m$ replacing $3$) are those which are 
better than $1/(2m)$-H\"{o}lder-continuous in the $L^{2}\left(  \Omega\right) 
$-sense, whether for Gaussian processes [Section \ref{GAUSS}], or non-Gaussian 
ones [Section \ref{NGC}]. Specifically, 
 
\begin{itemize} 
\item for $X$ \emph{Gaussian with stationary increments} (i.e. $\delta\left( 
s,t\right)  =\delta\left(  t-s\right)  $), for any odd integer $m\geq3$, 
$\left[  X,m\right]  =0$ if and only if $\delta\left(  r\right)  =o\left( 
r^{1/\left(  2m\right)  }\right)  $ for $r$ near $0$ [Theorem \ref{HomogGauss} 
on page \pageref{HomogGauss}]; 
 
\item for $X$ \emph{Gaussian} but not necessarily with stationary increments, 
for any odd integer $m\geq3$, $\left[  X,m\right]  =0$ if $\delta^{2}\left( 
s,s+r\right)  =o\left(  r^{1/(2m)}\right)  $ for $r$ near $0$ uniformly in 
$s$. [Theorem \ref{nonhomoggauss} on page \pageref{nonhomoggauss}; this holds 
under a technical non-explosion condition on the mixed partial derivative of 
$\delta^{2}$ near the diagonal; see Examples \ref{RL} and \ref{RLgen} on page 
\pageref{RL} for a wide class of Volterra-convolution-type Gaussian processes 
with non-stationary increments which satisfy the condition]. 
 
\item for $X$ \emph{non-Gaussian} based on a martingale $M,$ for any odd 
integer $m\geq3$, with $\Gamma\left(  t\right)  :=\left(  \mathbf{E}\left[ 
\left(  d\left[  M\right]  /dt\right)  ^{m}\right]  \right)  ^{1/(2m)}$ if it 
exists, we let $Z\left(  t\right)  :=\int_{0}^{T}\Gamma\left(  s\right) 
G\left(  t,s\right)  dW\left(  s\right)  $. This $Z$ is a Gaussian process; if 
it satisfies the conditions of Theorem \ref{HomogGauss} or Theorem 
\ref{nonhomoggauss}, then $\left[  X,m\right]  =0$. [Section \ref{NGC}, 
Theorem \ref{MartThm} on page \pageref{MartThm}; Proposition \ref{Ex} on page 
\pageref{Ex} provides examples of wide classes of martingales and kernels for 
which the assumptions of Theorem \ref{MartThm} are satisfied, with details on 
how to construct examples and study their regularity properties on page 
\ref{pex}.]. 
\end{itemize} 
 
Our results shows how broad a class of processes, based on martingale 
convolutions with only $m$ moments, one can construct which have zero odd 
$m$th variation, under conditions which are the same in terms of regularity as 
in the case of Gaussian processes with stationary increments, where we prove 
sharpness. Note that $X$ itself can be far from having the martingale 
property, just as it is generally far from standard Brownian motion in the 
Gaussian case. Our second main result is an application to weighted 
variations, generalized Stratonovich integration, and an It\^{o} formula 
[Section \ref{STOCH}.] 
 
\begin{itemize} 
\item Under the conditions of Theorem \ref{nonhomoggauss} (general Gaussian 
case), and an additional coercivity condition, for every bounded measurable 
function $g$ on $\mathbf{R}$, 
\[ 
\lim_{\varepsilon\rightarrow0}\frac{1}{\varepsilon^{2}}\mathbf{E}\left[ 
\left(  \int_{0}^{T}du\left(  X_{u+\varepsilon}-X_{u}\right)  ^{m}g\left( 
\frac{X_{u+\varepsilon}+X_{u}}{2}\right)  \right)  ^{2}\right]  =0. 
\] 
If $m=3$, by results in \cite{GNRV}, Theorem \ref{ForItoThm} implies that for 
any $f\in C^{6}\left(  \mathbf{R}\right)  $ and $t\in\lbrack0,T]$, the It\^{o} 
formula $f\left(  X_{t}\right)  =f\left(  X_{0}\right)  +\int_{0}^{t}% 
f^{\prime}\left(  X_{u}\right)  d^{\circ}X_{u}$ holds, where the integral is 
in the symmetric (generalized Stratonovich) sense. [Theorem \ref{ForItoThm} 
and its Corollary \ref{coroll}, on page \pageref{coroll}.] 
\end{itemize} 
 
Most of the proofs of our theorems are relegated to the Appendix [Section 
\ref{APP}]. 
 
\subsection{Relation with other recent work} 
 
The authors of the paper \cite{HNS} consider, as we do, stochastic processes 
which can be written as Volterra integrals with respect to martingales. Their 
\textquotedblleft fractional martingale\textquotedblright, which generalizes 
Riemann-Liouville fBm, is a special case of the processes we consider in 
Section \ref{NGC}, with $K\left(  t,s\right)  =\left(  t-s\right)  ^{H-1/2}$. 
The authors' motivation is to prove an analogue of the famous L\'evy 
characterization of Brownian motion as the only continuous square-integrable 
martingale with a quadratic variation equal to $t$. They provide similar 
necessary and sufficient conditions based on the $1/H$-variation for a process 
to be fBm. This is a different aspect of the theory than our motivation to 
study necessary and sufficient conditions for a process to have vanishing 
(odd) cubic variation, and its relation to stochastic calculus. The value 
$H=1/6$ is mentioned in the context of the stochastic heat equation driven by 
space-time white-noise, in which discrete trapezoidal sums converge in 
distribution (not in probability) to a conditionally independent Brownian 
motion: see \cite{BS} and \cite{NOT}. 
 
To find a similar motivation to ours, one may look at the recent result of 
\cite{NNT}, where the authors study the central and non-central behavior of 
weighted Hermite variations for fBm. Using the Hermite polynomial of order $m$ 
rather than the power-$m$ function, they show that the threshold value 
$H=1/\left(  2m\right)  $ poses an interesting open problem, since above this 
threshold (but below $H=1-1/\left(  2m\right)  $) one obtains Gaussian limits 
(these limits are conditionally Gaussian when weights are present, and can be 
represented as Wiener integrals with respect to an independent Brownian 
motion), while below the threshold, degeneracy occurs. The behavior at the 
threshold was worked out for $H=1/4,m=2$ in \cite{NNT}, boasting an exotic 
correction term with an independent Brownian motion, while the general open 
problem of Hermite variations with $H=1/\left(  2m\right)  $ was settled in 
\cite{NN}. More questions arise, for instance, with a similar result in 
\cite{N2} for $H=1/4$, but this time with bidimensional fBm, in which two 
independent Brownian motions are needed to characterize the exotic correction 
term. Compared to the above works, our work situates itself by 
 
\begin{itemize} 
\item establishing necessary and sufficient conditions for nullity of the 
$m$th power variation, around the threshold regularity value $H=1/(2m)$, for 
general Gaussian processes with stationary increments, showing in particular 
that self-similarity is not related to this nullity, and that the result works 
for all odd integers, thanks only to the problem's symmetries; 
 
\item showing that our method is able to consider processes that are far from 
Gaussian and still yield sharp sufficient conditions for nullity of odd power 
variations, since our base noise may be a generic martingale with only a few 
moments; our ability to prove an It\^{o} formula for such processes attests to 
our method's power. 
\end{itemize} 
 
\section{Definitions\label{DEF}} 
 
We recall our process $X$ defined for all $t\in\lbrack0,T]$ by (\ref{defX}). 
For any integer $m\geq2$, let the \emph{odd }$\varepsilon$\emph{-}$m$\emph{-th 
variation} of $X$ be defined by% 
\begin{equation} 
\lbrack X,m]_{\varepsilon}\left(  T\right)  :=\frac{1}{\varepsilon}\int 
_{0}^{T}ds\left\vert X\left(  s+\varepsilon\right)  -X\left(  s\right) 
\right\vert ^{m}\mbox{sgn}\left(  X\left(  s+\varepsilon\right)  -X\left( 
s\right)  \right)  . \label{defXm}% 
\end{equation} 
The odd variation is different from the absolute (or strong) variation because 
of the presence of the sign function, making the function $\left\vert 
x\right\vert ^{m}\mbox{sgn}\left(  x\right)  $ an odd function. In the sequel, 
in order to lighten the notation, we will write $\left(  x\right)  ^{m}$ for 
$\left\vert x\right\vert ^{m}\mbox{sgn}\left(  x\right)  $. We say that $X$ 
has \emph{zero odd }$m$\emph{-th variation} (in the mean-squared sense) if the 
limit $\lim_{\varepsilon\rightarrow0}[X,m]_{\varepsilon}\left(  T\right)  =0$ 
holds in $L^{2}\left(  \Omega\right)  $. 
 
The \emph{canonical metric} $\delta$ of a stochastic process $X$ is defined as 
the pseudo-metric on $[0,T]^{2}$ given by $\delta^{2}\left(  s,t\right) 
=\mathbf{E}\left[  \left(  X\left(  t\right)  -X\left(  s\right)  \right) 
^{2}\right]  $. The \emph{covariance function} of $X$ is defined by $Q\left( 
s,t\right)  =\mathbf{E}\left[  X\left(  t\right)  X\left(  s\right)  \right] 
$. The special case of a centered Gaussian process is of primary importance; 
then the process's entire distribution is characterized by $Q$, or alternately 
by $\delta$ and the variances $var\left(  X\left(  t\right)  \right) 
=Q\left(  t,t\right)  $, since we have $Q\left(  s,t\right)  =\frac{1}% 
{2}\left(  Q\left(  s,s\right)  +Q\left(  t,t\right)  -\delta^{2}\left( 
s,t\right)  \right)  $. We say that $\delta$ has \emph{stationary increments} 
if there exists a function on $[0,T]$ which we also denote by $\delta$ such 
that $\delta\left(  s,t\right)  =\delta\left(  \left\vert t-s\right\vert 
\right)  $. Below, we will refer to this situation as the \emph{stationary 
case}. This is in contrast to usual usage of this appellation, which is 
stronger, since for example in the Gaussian case, it refers to the fact that 
$Q\left(  s,t\right)  $ depends only on the difference $s-t$; this would not 
apply to, say, standard or fBm, while our definition does. In non-Gaussian 
settings, the usual way to interpret the \textquotedblleft 
stationary\textquotedblright\ property is to require that the processes 
$X\left(  t+\cdot\right)  $ and $X\left(  \cdot\right)  $ have the same law, 
which is typically much more restrictive than our definition. 
 
The goal of the next two sections is to define various general conditions 
under which a characterization of $\lim_{\varepsilon\rightarrow0}% 
[X,m]_{\varepsilon}\left(  T\right)  =0$ can be established. In particular, we 
aim to show that $X$ has zero odd $m$-th variation for well-behaved $M$'s and 
$G$'s if -- and in some cases only if --% 
\begin{equation} 
\delta\left(  s,t\right)  =o(\left\vert t-s\right\vert ^{1/\left(  2m\right) 
}). \label{condelta}% 
\end{equation}

\section{Gaussian case\label{GAUSS}} 
 
We assume that $X$ is centered Gaussian. Then we can write $X$ as in formula 
(\ref{defX}) with $M=W$ a standard Brownian motion. We have the following 
elementary result. 
 
\begin{lemma} 
\label{lemma1}If $m$ is an odd integer $\geq3$, we have $\mathbf{E}\left[ 
\left(  [X,m]_{\varepsilon}\left(  T\right)  \right)  ^{2}\right]  =\sum 
_{j=0}^{\left(  m-1\right)  /2}J_{j}$ where% 
\[ 
J_{j}:=\frac{1}{\varepsilon^{2}}\sum_{j=0}^{\left(  m-1\right)  /2}c_{j}% 
\int_{0}^{T}\int_{0}^{t}dtds\Theta^{\varepsilon}\left(  s,t\right) 
^{m-2j}\ Var\left[  X\left(  t+\varepsilon\right)  -X\left(  t\right) 
\right]  ^{j}\ Var\left[  X\left(  s+\varepsilon\right)  -X\left(  s\right) 
\right]  ^{j}, 
\] 
the $c_{j}$'s are positive constants depending only on $j$, and 
\[ 
\Theta^{\varepsilon}\left(  s,t\right)  :=\mathbf{E}\left[  \left(  X\left( 
t+\varepsilon\right)  -X\left(  t\right)  \right)  \left(  X\left( 
s+\varepsilon\right)  -X\left(  s\right)  \right)  \right]  . 
\] 
 
\end{lemma} 
 
Using $Q$ and $\delta$, $\Theta^{\varepsilon}\left(  s,t\right)  $ computes as 
the opposite of the planar increment of the canonical metric over the 
rectangle defined by its corners $\left(  s,t\right)  $ and $\left( 
s+\varepsilon,t+\varepsilon\right)  $:% 
\begin{equation} 
\Theta^{\varepsilon}\left(  s,t\right)  =\frac{1}{2}\left[  -\delta^{2}\left( 
t+\varepsilon,s+\varepsilon\right)  +\delta^{2}\left(  t,s+\varepsilon\right) 
+\delta^{2}\left(  s,t+\varepsilon\right)  -\delta^{2}\left(  s,t\right) 
\right]  =:-\frac{1}{2}\Delta_{\left(  s,t\right)  ;\left(  s+\varepsilon 
,t+\varepsilon\right)  }\delta^{2}. \label{defDelta}% 
\end{equation}

\subsection{The case of critical fBm} 
 
Before finding sufficient and possibly necessary conditions for various 
Gaussian processes to have zero cubic (or $m$th) variation, we discuss the 
threshold case for the cubic variation of fBm. Recall that when $X$ is fBm 
with parameter $H=1/6$, as mentioned in the Introduction, it is known from 
\cite[Theorem 4.1 part (2)]{GNRV} that $[X,3]_{\varepsilon}\left(  T\right)  $ 
converges in distribution to a non-degenerate normal law. However, there does 
not seem to be any place in the literature specifying whether the convergence 
may be any stronger than in distribution. We address this issue here. 
 
\begin{proposition} 
\label{ChaosConv}Let $X$ be an fBm with Hurst parameter $H=1/6$. Then $X$ does 
not have a cubic variation (in the mean-square sense), by which we mean that 
$[X,3]_{\varepsilon}\left(  T\right)  $ has no limit in $L^{2}\left( 
\Omega\right)  $ as $\varepsilon\rightarrow0$. In fact more is true: 
$[X,3]_{\varepsilon}\left(  T\right)  $ has no limit in probability as 
$\varepsilon\rightarrow0$. 
\end{proposition} 
 
In order to prove the proposition, we study the Wiener chaos representation 
and moments of $[X,3]_{\varepsilon}\left(  T\right)  $ when $X$ is fBm; $X$ is 
given by (\ref{defX}) where $W$ is Brownian motion and the kernel $G$ is 
well-known (see Chapters 1 and 5 of the textbook \cite{Nbook}). 
 
\begin{lemma} 
\label{X3eExact}Fix $\varepsilon>0$. Let $\Delta G_{s}\left(  u\right) 
:=G\left(  s+\varepsilon,u\right)  -G\left(  s,u\right)  $. Then 
$[X,3]_{\varepsilon}\left(  T\right)  =\mathcal{I}_{1}+\mathcal{I}_{3}$ where% 
\begin{align} 
\mathcal{I}_{1}  &  :=\frac{3}{\varepsilon}\int_{0}^{T}ds\int_{0}^{T}\Delta 
G_{s}\left(  u\right)  dW\left(  u\right)  \left(  \int_{0}^{T}\left\vert 
\Delta G_{s}\left(  v\right)  \right\vert ^{2}dv\right)  ,\label{firstchaos}\\ 
\mathcal{I}_{3}  &  :=\frac{6}{\varepsilon}\int_{0}^{T}dW\left(  s_{3}\right) 
\int_{0}^{s_{3}}dW\left(  s_{2}\right)  \int_{0}^{s_{2}}dW\left( 
s_{1}\right)  \int_{0}^{T}\left[  \prod_{k=1}^{3}\Delta G_{s}\left( 
s_{k}\right)  \right]  ds. \label{thirdchaos}% 
\end{align} 
 
\end{lemma} 
 
The above lemma indicates the Wiener chaos decomposition of 
$[X,3]_{\varepsilon}\left(  T\right)  $ into the term $\mathcal{I}_{1}$ of 
line (\ref{firstchaos}) which is in the first Wiener chaos (i.e. a Gaussian 
term), and the term $\mathcal{I}_{3}$ of line (\ref{thirdchaos}), in the third 
Wiener chaos. The next two lemmas contain information on the behavior of each 
of these two terms, as needed to prove Proposition \ref{ChaosConv}. 
 
\begin{lemma} 
\label{I1pro}$\mathcal{I}_{1}$ converges to $0$ in $L^{2}\left( 
\Omega\right)  $ as $\varepsilon\rightarrow0$. 
\end{lemma} 
 
\begin{lemma} 
\label{I3pro}$\mathcal{I}_{3}$ is bounded in $L^{2}\left(  \Omega\right)  $ 
for all $\varepsilon>0$, and does not converge in $L^{2}\left(  \Omega\right) 
$ as $\varepsilon\rightarrow0$. 
\end{lemma} 
 
\begin{proof} 
[Proof of Proposition \ref{ChaosConv}]We prove the proposition by 
contradiction. Assume $[X,3]_{\varepsilon}\left(  T\right)  $ converges in 
probability. For any $p>2$, there exists $c_{p}$ depending only on $p$ such 
that $\mathbf{E}\left[  \left\vert \mathcal{I}_{1}\right\vert ^{p}\right] 
\leq c_{p}\left(  \mathbf{E}\left[  \left\vert \mathcal{I}_{1}\right\vert 
^{2}\right]  \right)  ^{p/2}$ and $\mathbf{E}\left[  \left\vert \mathcal{I}% 
_{3}\right\vert ^{p}\right]  \leq c_{p}\left(  \mathbf{E}\left[  \left\vert 
\mathcal{I}_{3}\right\vert ^{2}\right]  \right)  ^{p/2}$; this is a general 
fact about random variables in fixed Wiener chaos, and can be proved directly 
using Lemma \ref{X3eExact} and the Burkh\"{o}lder-Davis-Gundy inequalities. 
Also see \cite{Borell}. Therefore, since we have $\sup_{\varepsilon 
>0}(\mathbf{E}\left[  \left\vert \mathcal{I}_{1}\right\vert ^{2}\right] 
+\mathbf{E}\left[  \left\vert \mathcal{I}_{3}\right\vert ^{2}\right] 
)<\infty$ by Lemmas \ref{I1pro} and \ref{I3pro}, we also get $\sup 
_{\varepsilon>0}(\mathbf{E}\left[  \left\vert \mathcal{I}_{1}+\mathcal{I}% 
_{3}\right\vert ^{p}\right]  )<\infty$ for any $p$. Therefore, by uniform 
integrability, $[X,3]_{\varepsilon}\left(  T\right)  =\mathcal{I}% 
_{1}+\mathcal{I}_{3}$ converges in $L^{2}\left(  \Omega\right)  $. In 
$L^{2}\left(  \Omega\right)  $, the terms $\mathcal{I}_{1}$ and $\mathcal{I}% 
_{3}$ are orthogonal. Therefore, $\mathcal{I}_{1}$ and $\mathcal{I}_{3}$ must 
converge in $L^{2}\left(  \Omega\right)  $ separately. This contradicts the 
non-convergence of $\mathcal{I}_{3}$ in $L^{2}\left(  \Omega\right)  $ 
obtained in Lemma \ref{I3pro}. Thus $[X,3]_{\varepsilon}\left(  T\right)  $ 
does not converge in probability. 
\end{proof} 
 
\subsection{The case of stationary increments\label{HomogGaussSect}} 
 
We prove a necessary and sufficient condition for having a zero odd $m$-th 
variation for Gaussian processes with stationary increments. 
 
\begin{theorem} 
\label{HomogGauss}Let $m>1$ be an odd integer. Let $X$ be a centered Gaussian 
process on $[0,T]$ with stationary increments; its canonical metric is% 
\[ 
\delta^{2}\left(  s,t\right)  :=\mathbf{E}\left[  \left(  X\left(  t\right) 
-X\left(  s\right)  \right)  ^{2}\right]  =\delta^{2}\left(  \left\vert 
t-s\right\vert \right) 
\] 
where the univariate function $\delta^{2}$ is assumed to be increasing and 
concave on $[0,T]$. Then $X$ has zero $m$th variation if and only if 
$\delta\left(  r\right)  =o\left(  r^{1/\left(  2m\right)  }\right)  $. 
\end{theorem} 
 
\begin{proof} 
\noindent\emph{Step 0: setup.} The derivative $d\delta^{2}$ of $\delta^{2}$, 
in the sense of measures, is positive and bounded on $[0,T]$. By stationarity, 
$Var\left[  X\left(  t+\varepsilon\right)  -X\left(  t\right)  \right] 
=\delta^{2}\left(  \varepsilon\right)  .$ Using the notation in Lemma 
\ref{lemma1}, we get% 
\[ 
J_{j}=\varepsilon^{-2}\delta^{4j}\left(  \varepsilon\right)  c_{j}\int_{0}% 
^{T}dt\int_{0}^{t}ds\Theta^{\varepsilon}\left(  s,t\right)  ^{m-2j}. 
\] 
\vspace{0.1in} 
 
\noindent\emph{Step 1: diagonal. }We define the $\varepsilon$-diagonal 
$D_{\varepsilon}:=\left\{  0\leq t-\varepsilon<s<t\leq T\right\}  $. Trivially 
using the Cauchy-Schwarz's inequality,% 
\[ 
\left\vert \Theta^{\varepsilon}\left(  s,t\right)  \right\vert \leq 
\sqrt{Var\left[  X\left(  t+\varepsilon\right)  -X\left(  t\right)  \right] 
Var\left[  X\left(  s+\varepsilon\right)  -X\left(  s\right)  \right] 
}=\delta^{2}\left(  \varepsilon\right)  . 
\] 
Hence, according to Lemma \ref{lemma1}, the diagonal portion $\sum 
_{j=0}^{\left(  m-1\right)  /2}J_{j,D_{\varepsilon}}$ of $\mathbf{E}\left[ 
\left(  [X,m]_{\varepsilon}\left(  T\right)  \right)  ^{2}\right]  $ can be 
bounded above, in absolute value, as:% 
\begin{align*} 
\left\vert \sum_{j=0}^{\left(  m-1\right)  /2}J_{j,D_{\varepsilon}% 
}\right\vert  &  :=\left\vert \sum_{j=0}^{\left(  m-1\right)  /2}% 
\varepsilon^{-2}\delta^{4j}\left(  \varepsilon\right)  c_{j}\int_{\varepsilon 
}^{T}dt\int_{t-\varepsilon}^{t}ds\Theta^{\varepsilon}\left(  s,t\right) 
^{m-2j}\right\vert \\ 
&  \leq\frac{1}{\varepsilon^{2}}\sum_{j=0}^{\left(  m-1\right)  /2}c_{j}% 
\int_{\varepsilon}^{T}dt\int_{t-\varepsilon}^{t}ds\delta^{2m}\left( 
\varepsilon\right)  \leq cst\cdot\varepsilon^{-1}\delta^{2m}\left( 
\varepsilon\right) 
\end{align*} 
where $cst$ denotes a constant (here depending only on $\delta$ and $m$) whose 
value may change in the remainder of the article's proofs. The hypothesis on 
$\delta^{2}$ implies that the above converges to $0$ as $\varepsilon$ tends to 
$0$.\vspace{0.1in} 
 
\noindent\emph{Step 2: small }$t$\emph{ term . }The term for $t\in 
\lbrack0,\varepsilon]$ and any $s\in\lbrack0,t]$ can be dealt with similarly, 
and is of a smaller order than the one in Step 1. Specifically we have% 
\[ 
\left\vert J_{j,S}\right\vert :=\varepsilon^{-2}\delta^{4j}\left( 
\varepsilon\right)  c_{j}\left\vert \int_{0}^{\varepsilon}dt\int_{0}% 
^{t}ds\Theta^{\varepsilon}\left(  s,t\right)  ^{m-2j}\right\vert 
\leq\varepsilon^{-2}\delta^{4j}\left(  \varepsilon\right)  c_{j}% 
\delta^{2\left(  m-2j\right)  }\left(  \varepsilon\right)  \varepsilon 
^{2}=c_{j}\delta^{2m}\left(  \varepsilon\right)  , 
\] 
which converges to $0$ like $o\left(  \varepsilon\right)  $.\vspace{0.1in} 
 
\noindent\emph{Step 3: off-diagonal. }By stationarity, from (\ref{defDelta}), 
for any $s,t$ in the $\varepsilon$-off diagonal set $OD_{\varepsilon 
}:=\left\{  0\leq s<t-\varepsilon<t\leq T\right\}  $,% 
\begin{align} 
\Theta^{\varepsilon}\left(  s,t\right)   &  =\left(  \delta^{2}\left( 
t-s+\varepsilon\right)  -\delta^{2}\left(  t-s\right)  \right)  -\left( 
\delta^{2}\left(  t-s\right)  -\delta^{2}\left(  t-s-\varepsilon\right) 
\right) \nonumber\\ 
&  =\int_{t-s}^{t-s+\varepsilon}d\delta^{2}\left(  r\right)  -\int 
_{t-s-\varepsilon}^{t-s}d\delta^{2}\left(  r\right)  . \label{Thetadiffdelta}% 
\end{align} 
By the concavity of $\delta^{2}$, we see that $\Theta^{\varepsilon}\left( 
s,t\right)  $ is negative in $OD_{\varepsilon}$. According to Lemma 
\ref{lemma1}, the off-diagonal portion $\sum_{j=0}^{\left(  m-1\right) 
/2}J_{j,OD_{\varepsilon}}$ of $\mathbf{E}\left[  \left(  [X,m]_{\varepsilon 
}\left(  T\right)  \right)  ^{2}\right]  $ is precisely equal to,% 
\[ 
\sum_{j=0}^{\left(  m-1\right)  /2}J_{j,OD_{\varepsilon}}:=\sum_{j=0}^{\left( 
m-1\right)  /2}\varepsilon^{-2}\delta^{4j}\left(  \varepsilon\right) 
c_{j}\int_{\varepsilon}^{T}dt\int_{0}^{t-\varepsilon}ds\Theta^{\varepsilon 
}\left(  s,t\right)  ^{m-2j}. 
\] 
The negativity of $\Theta^{\varepsilon}$on $OD_{\varepsilon}$, odd power 
$m-2j$, and positivity of all other factors above implies that the entire 
off-diagonal contribution to $\mathbf{E}\left[  \left(  [X,m]_{\varepsilon 
}\left(  T\right)  \right)  ^{2}\right]  $ is negative. Combining this with 
the results of Steps 1 and 2, we obtain that% 
\[ 
\mathbf{E}\left[  \left(  [X,m]_{\varepsilon}\left(  T\right)  \right) 
^{2}\right]  \leq cst\cdot\varepsilon^{-1}\delta^{2m}\left(  2\varepsilon 
\right) 
\] 
which implies the sufficient condition in the theorem.\vspace{0.1in} 
 
\noindent\emph{Step 4: necessary condition.} The proof of this part is more 
delicate than the above: it requires an excellent control of the off-diagonal 
term, since it is negative and turns out to be of the same order of magnitude 
as the diagonal term. We spell out the proof here for $m=3$. The general case 
is similar, and is left to the reader.\vspace{0.1in} 
 
\noindent\emph{Step 4.1: positive representation.} The next elementary lemma 
(see the product formula in \cite[Prop. 1.1.3]{Nbook}, or \cite[Thm 
9.6.9]{HHK}) uses the following chaos integral notation: for any 
$n\in\mathbf{N}$, for $g\in L^{2}\left(  [0,T]^{n}\right)  $, $g$ symmetric in 
its $n$ variables, then $I_{n}\left(  g\right)  $ is the multiple Wiener 
integral of $g$ over $[0,T]^{n}$ with respect to $W$. 
 
\begin{lemma} 
\label{cubechaos}Let $f\in L^{2}\left(  [0,T]\right)  $. Then $I_{1}\left( 
f\right)  ^{3}=3\left\vert f\right\vert _{L^{2}\left(  [0,T]\right)  }% 
^{2}I_{1}\left(  f\right)  +I_{3}\left(  f\otimes f\otimes f\right)  $ 
\end{lemma} 
 
Using this lemma, as well as definitions (\ref{defX}) and (\ref{defXm}), 
recalling the notation $\Delta G_{s}\left(  u\right)  :=G\left( 
s+\varepsilon,u\right)  -G\left(  s,u\right)  $ already used in Lemma 
\ref{X3eExact}, and exploiting the fact that the covariance of two multiple 
Wiener integrals of different orders is $0$, we can write 
\begin{align*} 
&  \mathbf{E}\left[  \left(  [X,3]_{\varepsilon}\left(  T\right)  \right) 
^{2}\right]  =\frac{9}{\varepsilon^{2}}\int_{0}^{T}ds\int_{0}^{T}% 
dt\mathbf{E}\left[  I_{1}\left(  \Delta G_{s}\right)  I_{1}\left(  \Delta 
G_{t}\right)  \right]  \left\vert \Delta G_{s}\right\vert _{L^{2}\left( 
[0,T]\right)  }^{2}\left\vert \Delta G_{t}\right\vert _{L^{2}\left( 
[0,T]\right)  }^{2}\\ 
&  +\frac{1}{\varepsilon^{2}}\int_{0}^{T}ds\int_{0}^{T}dt\mathbf{E}\left[ 
I_{3}\left(  \left(  \Delta G_{s}\right)  ^{\otimes3}\right)  I_{3}\left( 
\left(  \Delta G_{t}\right)  ^{\otimes3}\right)  \right]  . 
\end{align*} 
Now we use the fact that $\mathbf{E}\left[  I_{3}\left(  h\right) 
I_{3}\left(  \ell\right)  \right]  =\left\langle h,\ell\right\rangle 
_{L^{2}\left(  [0,T]^{3}\right)  }$, plus the fact that in our stationary 
situation $\left\vert \Delta G_{s}\right\vert _{L^{2}\left(  [0,T]\right) 
}^{2}=\delta^{2}\left(  \varepsilon\right)  $ for any $s$. Hence the above 
equals% 
\begin{align*} 
&  \frac{9\delta^{4}\left(  \varepsilon\right)  }{\varepsilon^{2}}\int_{0}% 
^{T}ds\int_{0}^{T}dt\left\langle \Delta G_{s},\Delta G_{t}\right\rangle 
_{L^{2}\left(  [0,T]\right)  }+\frac{1}{\varepsilon^{2}}\int_{0}^{T}ds\int 
_{0}^{T}dt\left\langle \left(  \Delta G_{s}\right)  ^{\otimes3},\left(  \Delta 
G_{t}\right)  ^{\otimes3}\right\rangle _{L^{2}\left(  [0,T]^{3}\right)  }\\ 
&  =\frac{9\delta^{4}\left(  \varepsilon\right)  }{\varepsilon^{2}}\int 
_{0}^{T}ds\int_{0}^{T}dt\int_{0}^{T}du\Delta G_{s}\left(  u\right)  \Delta 
G_{t}\left(  u\right)  +\frac{1}{\varepsilon^{2}}\int_{0}^{T}ds\int_{0}^{T}dt% 
%TCIMACRO{\diiint \limits_{[0,T]^{3}}}% 
%BeginExpansion 
{\displaystyle\iiint\limits_{[0,T]^{3}}} 
%EndExpansion 
\prod_{i=1}^{3}\left(  du_{i}\Delta G_{s}\left(  u_{i}\right)  \Delta 
G_{t}\left(  u_{i}\right)  \right) \\ 
&  =\frac{9\delta^{4}\left(  \varepsilon\right)  }{\varepsilon^{2}}\int 
_{0}^{T}du\left\vert \int_{0}^{T}ds\Delta G_{s}\left(  u\right)  \right\vert 
^{2}+\frac{1}{\varepsilon^{2}}% 
%TCIMACRO{\diiint \limits_{[0,T]^{3}}}% 
%BeginExpansion 
{\displaystyle\iiint\limits_{[0,T]^{3}}} 
%EndExpansion 
du_{1}\ du_{2}\ du_{3}\left\vert \int_{0}^{T}ds\prod_{i=1}^{3}\left(  \Delta 
G_{s}\left(  u_{i}\right)  \right)  \right\vert ^{2}. 
\end{align*} 
\vspace{0.1in} 
 
\noindent\emph{Step 4.2: }$J_{1}$\emph{ as a lower bound}. The above 
representation is extremely useful because it turns out, as one readily 
checks, that of the two summands in the last expression above, the first is 
what we called $J_{1}$ and the second is $J_{0}$, and we can now see that both 
these terms are positive, which was not at all obvious before, since, as we 
recall, the off-diagonal contribution to either term is negative by our 
concavity assumption. Nevertheless, we may now have a lower bound on the 
$\varepsilon$-variation by finding a lower bound for the term $J_{1}$ alone. 
Reverting to our method of separating diagonal and off-diagonal terms, and 
recalling by Step 2 that we can restrict $t\geq\varepsilon$, we have% 
\begin{align*} 
J_{1}  &  =\frac{9\delta^{4}\left(  \varepsilon\right)  }{\varepsilon^{2}% 
}2\int_{\varepsilon}^{T}dt\int_{0}^{t}ds\int_{0}^{T}du\Delta G_{s}\left( 
u\right)  \Delta G_{t}\left(  u\right)  =\frac{9\delta^{4}\left( 
\varepsilon\right)  }{\varepsilon^{2}}2\int_{\varepsilon}^{T}dt\int_{0}% 
^{t}ds\Theta_{\varepsilon}\left(  s,t\right) \\ 
&  =\frac{9\delta^{4}\left(  \varepsilon\right)  }{\varepsilon^{2}}% 
\int_{\varepsilon}^{T}dt\int_{0}^{t}ds\left(  \delta^{2}\left( 
t-s+\varepsilon\right)  -\delta^{2}\left(  t-s\right)  -\left(  \delta 
^{2}\left(  t-s\right)  -\delta^{2}\left(  \left\vert t-s-\varepsilon 
\right\vert \right)  \right)  \right) \\ 
&  =J_{1,D}+J_{1,OD}% 
\end{align*} 
where, performing the change of variables $t-s\mapsto s$% 
\begin{align*} 
J_{1,D}  &  :=\frac{9\delta^{4}\left(  \varepsilon\right)  }{\varepsilon^{2}% 
}\int_{\varepsilon}^{T}dt\int_{0}^{\varepsilon}ds\left(  \delta^{2}\left( 
s+\varepsilon\right)  -\delta^{2}\left(  s\right)  -\left(  \delta^{2}\left( 
s\right)  -\delta^{2}\left(  \varepsilon-s\right)  \right)  \right) \\ 
J_{1,OD}  &  :=\frac{9\delta^{4}\left(  \varepsilon\right)  }{\varepsilon^{2}% 
}\int_{\varepsilon}^{T}dt\int_{\varepsilon}^{t}ds\left(  \delta^{2}\left( 
s+\varepsilon\right)  -\delta^{2}\left(  s\right)  -\left(  \delta^{2}\left( 
s\right)  -\delta^{2}\left(  s-\varepsilon\right)  \right)  \right)  . 
\end{align*} 
\vspace{0.1in} 
 
\noindent\emph{Step 4.3: Upper bound on }$\left\vert J_{1,OD}\right\vert $. We 
rewrite the planar increments of $\delta^{2}$ as in (\ref{Thetadiffdelta}) to 
show what cancellations occur: with the change of variable $s^{\prime 
}:=t-s-\varepsilon$, we get $-\Theta^{\varepsilon}\left(  s,t\right) 
=-\int_{s^{\prime}}^{s^{\prime}+\varepsilon}d\delta^{2}\left(  r\right) 
+\int_{s^{\prime}-\varepsilon}^{s^{\prime}}d\delta^{2}\left(  r\right)  $, 
and 
\begin{align*} 
\int_{\varepsilon}^{T}dt\int_{0}^{t-\varepsilon}ds\left(  -\Theta 
^{\varepsilon}\left(  s,t\right)  \right)   &  =\int_{\varepsilon}% 
^{T}dt\left[  \int_{\varepsilon}^{t}ds^{\prime}\int_{s^{\prime}-\varepsilon 
}^{s^{\prime}}d\delta^{2}\left(  r\right)  -\int_{\varepsilon}^{t}ds^{\prime 
}\int_{s^{\prime}}^{s^{\prime}+\varepsilon}d\delta^{2}\left(  r\right) 
\right] \\ 
&  =\int_{\varepsilon}^{T}dt\left[  \int_{0}^{t-\varepsilon}ds^{\prime\prime 
}\int_{s^{\prime\prime}}^{s^{\prime\prime}+\varepsilon}d\delta^{2}\left( 
r\right)  -\int_{\varepsilon}^{t}ds^{\prime}\int_{s^{\prime}}^{s^{\prime 
}+\varepsilon}d\delta^{2}\left(  r\right)  \right] \\ 
&  =\int_{\varepsilon}^{T}dt\left[  \int_{0}^{\varepsilon}ds^{\prime\prime 
}\int_{s^{\prime\prime}}^{s^{\prime\prime}+\varepsilon}d\delta^{2}\left( 
r\right)  -\int_{t-\varepsilon}^{t}ds^{\prime}\int_{s^{\prime}}^{s^{\prime 
}+\varepsilon}d\delta^{2}\left(  r\right)  \right] 
\end{align*} 
where we also used the change $s^{\prime\prime}:=s^{\prime}-\varepsilon.$ Thus 
we have 
\[ 
J_{1,OD}=\frac{9\delta^{4}\left(  \varepsilon\right)  }{\varepsilon^{2}}% 
\int_{\varepsilon}^{T}dt\left[  \int_{t-\varepsilon}^{t}ds\int_{s}% 
^{s+\varepsilon}d\delta^{2}\left(  r\right)  -\int_{0}^{\varepsilon}ds\int 
_{s}^{s+\varepsilon}d\delta^{2}\left(  r\right)  \right]  =:K_{1}+K_{2}. 
\] 
We can already see that $K_{1}\geq0$ and $K_{2}\leq0$, so it is only necessary 
to find an upper bound on $\left\vert K_{2}\right\vert $; by Fubini on 
$\left(  r,s\right)  $, the integrand in $K_{2}$ is calculated as% 
\[ 
\int_{0}^{\varepsilon}ds\int_{s}^{s+\varepsilon}d\delta^{2}\left(  r\right) 
=-\int_{0}^{\varepsilon}\delta^{2}\left(  r\right)  dr+\int_{\varepsilon 
}^{2\varepsilon}\delta^{2}\left(  r\right)  dr. 
\] 
In particular, because $\left\vert K_{1}\right\vert \ll\left\vert 
K_{2}\right\vert $ and $\delta^{2}$ is increasing, we get% 
\begin{equation} 
\left\vert J_{1,OD}\right\vert \leq\frac{9\left(  T-\varepsilon\right) 
\delta^{4}\left(  \varepsilon\right)  }{\varepsilon^{2}}\left(  \int 
_{\varepsilon}^{2\varepsilon}\delta^{2}\left(  r\right)  dr-\int 
_{0}^{\varepsilon}\delta^{2}\left(  r\right)  dr\right)  . \label{J1ODlater}% 
\end{equation} 
\vspace{0.1in} 
 
\noindent\emph{Step 4.4: Lower bound on }$J_{1,D}$. Note first that 
\[ 
\int_{0}^{\varepsilon}ds\left(  \delta^{2}\left(  s\right)  -\delta^{2}\left( 
\varepsilon-s\right)  \right)  =\int_{0}^{\varepsilon}ds\ \delta^{2}\left( 
s\right)  -\int_{0}^{\varepsilon}ds\ \delta^{2}\left(  \varepsilon-s\right) 
=0. 
\] 
Therefore% 
\[ 
J_{1,D}=\frac{9\delta^{4}\left(  \varepsilon\right)  }{\varepsilon^{2}}% 
\int_{\varepsilon}^{T}dt\int_{0}^{\varepsilon}ds\left(  \delta^{2}\left( 
s+\varepsilon\right)  -\delta^{2}\left(  s\right)  \right)  =\frac{9\delta 
^{4}\left(  \varepsilon\right)  }{\varepsilon^{2}}\left(  T-\varepsilon 
\right)  \int_{0}^{\varepsilon}ds\int_{s}^{s+\varepsilon}d\delta^{2}\left( 
r\right)  . 
\] 
We can also perform a Fubini on the integral in $J_{1,D}$, easily obtaining% 
\[ 
J_{1,D}=\frac{9\delta^{4}\left(  \varepsilon\right)  }{\varepsilon^{2}}\left( 
T-\varepsilon\right)  \left(  \varepsilon\delta^{2}\left(  2\varepsilon 
\right)  -\int_{0}^{\varepsilon}\delta^{2}\left(  r\right)  dr\right)  . 
\]

\noindent\emph{Step 4.5: conclusion.} We may now compare $J_{1,D}$ and 
$\left\vert J_{1,OD}\right\vert $: by the results of Steps 4.1 and 4.2, 
\begin{align*} 
&  J_{1}=J_{1,D}-\left\vert J_{1,OD}\right\vert \geq\frac{9\delta^{4}\left( 
\varepsilon\right)  }{\varepsilon^{2}}\left(  T-\varepsilon\right)  \left( 
\varepsilon\delta^{2}\left(  2\varepsilon\right)  -\int_{0}^{\varepsilon 
}\delta^{2}\left(  r\right)  dr\right) \\ 
&  -\frac{9\delta^{4}\left(  \varepsilon\right)  }{\varepsilon^{2}}\left( 
T-\varepsilon\right)  \left(  \int_{\varepsilon}^{2\varepsilon}\delta 
^{2}\left(  r\right)  dr-\int_{0}^{\varepsilon}\delta^{2}\left(  r\right) 
dr\right)  =\frac{9\delta^{4}\left(  \varepsilon\right)  }{\varepsilon^{2}% 
}\left(  T-\varepsilon\right)  \int_{\varepsilon}^{2\varepsilon}\left( 
\delta^{2}\left(  2\varepsilon\right)  -\delta^{2}\left(  r\right)  \right) 
dr. 
\end{align*} 
When $\delta$ is in the H\"{o}lder scale $\delta\left(  r\right)  =r^{H}$, the 
above quantity is obviously commensurate with $\delta^{6}\left( 
\varepsilon\right)  /\varepsilon$, which implies the desired result, but in 
order to be sure we are treating all cases, we now present a general proof 
which only relies on the fact that $\delta^{2}$ is increasing and concave. 
 
Below we use the notation $\left(  \delta^{2}\right)  ^{\prime}$ for the 
density of $d\delta^{2}$, which exists a.e. since $\delta^{2}$ is concave. The 
mean value theorem and the concavity of $\delta^{2}$ then imply that for any 
$r\in\lbrack\varepsilon,2\varepsilon]$, 
\[ 
\delta^{2}\left(  2\varepsilon\right)  -\delta^{2}\left(  r\right) 
\geq\left(  2\varepsilon-r\right)  \inf_{[\varepsilon,2\varepsilon]}\left( 
\delta^{2}\right)  ^{\prime}=\left(  2\varepsilon-r\right)  \left(  \delta 
^{2}\right)  ^{\prime}\left(  2\varepsilon\right)  . 
\] 
Thus we can write% 
\begin{align*} 
J_{1}  &  \geq9(T-\varepsilon)\varepsilon^{-1}\delta^{4}\left(  \varepsilon 
\right)  \left(  \delta^{2}\right)  ^{\prime}\left(  2\varepsilon\right) 
\int_{\varepsilon}^{2\varepsilon}\left(  2\varepsilon-r\right) 
dr=9(T-\varepsilon)\varepsilon^{-1}\delta^{4}\left(  \varepsilon\right) 
\left(  \delta^{2}\right)  ^{\prime}\left(  2\varepsilon\right) 
\varepsilon^{2}/2\\ 
&  \geq cst\cdot\delta^{4}\left(  \varepsilon\right)  \cdot\left(  \delta 
^{2}\right)  ^{\prime}\left(  2\varepsilon\right)  . 
\end{align*} 
Since $\delta^{2}$ is concave, and $\delta\left(  0\right)  =0$, we have 
$\delta^{2}\left(  \varepsilon\right)  \geq\delta^{2}\left(  2\varepsilon 
\right)  /2$. Hence, with the notation $f\left(  x\right)  =\delta^{2}\left( 
2x\right)  $, we have% 
\[ 
J_{1}\geq cst\cdot f^{2}\left(  \varepsilon\right)  f^{\prime}\left( 
\varepsilon\right)  =cst\cdot\left(  f^{3}\right)  ^{\prime}\left( 
\varepsilon\right)  . 
\] 
Therefore we have that $\lim_{\varepsilon\rightarrow0}\left(  f^{3}\right) 
^{\prime}\left(  \varepsilon\right)  =0$. We prove this implies $\lim 
_{\varepsilon\rightarrow0}\varepsilon^{-1}f^{3}\left(  \varepsilon\right) 
=0$. Indeed, fix $\eta>0$; then there exists $\varepsilon_{\eta}>0$ such that 
for all $\varepsilon\in(0,\varepsilon_{\eta}]$, $0\leq\left(  f^{3}\right) 
^{\prime}\left(  \varepsilon\right)  \leq\eta$ (we used the positivity of 
$\left(  \delta^{2}\right)  ^{\prime}$). Hence, also using $f\left(  0\right) 
=0$, for any $\varepsilon\in(0,\varepsilon_{\eta}]$, 
\[ 
0\leq\frac{f^{3}\left(  \varepsilon\right)  }{\varepsilon}=\frac 
{1}{\varepsilon}\int_{0}^{\varepsilon}\left(  f^{3}\right)  ^{\prime}\left( 
x\right)  dx\leq\frac{1}{\varepsilon}\int_{0}^{\varepsilon}\eta dx=\eta. 
\] 
This proves that $\lim_{\varepsilon\rightarrow0}\varepsilon^{-1}f^{3}\left( 
\varepsilon\right)  =0$, which is equivalent to the announced necessary 
condition, and finishes the proof of the theorem. 
\end{proof} 
 
\subsection{Non-stationary case\label{NonHomogGaussSect}} 
 
The concavity and stationarity assumptions were used heavily above for the 
proof of the necessary condition in Theorem \ref{HomogGauss}. We now show they 
can be considerably weakened while still resulting in a sufficient condition: 
we only need a weak uniformity condition on the variances, coupled with a 
natural bound on the second-derivative measure of $\delta^{2}$. 
 
\begin{theorem} 
\label{nonhomoggauss}Let $m>1$ be an odd integer. Let $X$ be a centered 
Gaussian process on $[0,T]$ with canonical metric% 
\[ 
\delta^{2}\left(  s,t\right)  :=\mathbf{E}\left[  \left(  X\left(  t\right) 
-X\left(  s\right)  \right)  ^{2}\right]  . 
\] 
Define a univariate function on $[0,T]$, also denoted by $\delta^{2}$, via% 
\[ 
\delta^{2}\left(  r\right)  :=\sup_{s\in\lbrack0,T]}\delta^{2}\left( 
s,s+r\right)  , 
\] 
and assume that for $r$ near $0$,% 
\begin{equation} 
\delta\left(  r\right)  =o\left(  r^{1/2m}\right)  . \label{nhdeltacond}% 
\end{equation} 
Assume that, in the sense of distributions, the derivative $\partial\delta 
^{2}/\left(  \partial s\partial t\right)  $ is a finite signed $\sigma$ finite 
measure $\mu$ on $[0,T]^{2} - \Delta$ where $\Delta$ is the diagonal $\{(s,s) 
\vert s \in[0,T]\}$. Denote the off-diagonal simplex by $OD=\{\left( 
s,t\right)  :0\leq s\leq t-\varepsilon\leq T\}$; assume $\mu$ satisfies, for 
some constant $c$ and for all $\varepsilon$ small enough,% 
\begin{equation} 
\left\vert \mu\right\vert \left(  OD\right)  \leq c\varepsilon^{-(m-1)/m}, 
\label{nhmubirdiecond}% 
\end{equation} 
where $\left\vert \mu\right\vert $ is the total variation measure of $\mu$. 
Then $X$ has zero $m$th variation. 
\end{theorem} 
 
\begin{example} 
\label{RL} A typical situation covered by the above theorem is that of the 
Riemann-Liouville fBm $B^{H,RL}$ and similar non-stationary processes. The 
process $B^{H,RL}$ is defined by $B^{H,RL}\left(  t\right)  =\int_{0}% 
^{t}\left(  t-s\right)  ^{H-1/2}dW\left(  s\right)  $; it differs from the 
standard fBm by a bounded variation process, and as such it has zero $m$th 
variation for any $H>1/(2m)$. This can also be obtained via our theorem, 
because $B^{H,RL}$ is a member of the class of Gaussian processes whose 
canonical metric satisfies% 
\begin{equation} 
\left\vert t-s\right\vert ^{H}\leq\delta\left(  s,t\right)  \leq2\left\vert 
t-s\right\vert ^{H}. \label{ERL}% 
\end{equation} 
(see \cite{MV}). For any process satisfying (\ref{ERL}), our theorem's 
condition on variances is equivalent to $H>1/\left(  2m\right)  $, while for 
the other condition, a direct computation yields $\mu\left(  dsdt\right) 
/\left(  dsdt\right)  \asymp\left\vert t-s\right\vert ^{2H-2}dsdt$ off the 
diagonal, and therefore, for $H<1/2$,% 
\[ 
\mu\left(  OD\right)  =\left\vert \mu\right\vert \left(  OD\right)  \asymp 
\int_{0}^{T}\int_{\varepsilon}^{t}s^{2H-2}dsdt\asymp\varepsilon^{2H-1}. 
\] 
This quantity is bounded above by $\varepsilon^{-1+1/m}$ as soon as 
$H\geq1/\left(  2m\right)  $, of course, so the strict inequality is 
sufficient to apply the theorem and conclude that $B^{H,RL}$ all other 
processes satisfying (\ref{ERL}) have zero $m$th variation. 
\end{example} 
 
\begin{example} 
\label{RLgen}One can generalize Example \ref{RL} to any Gaussian process with 
a Volterra-convolution kernel: let $\gamma^{2}$ be a univariate increasing 
concave function, differentiable everywhere except possibly at $0$, and define% 
\begin{equation} 
X\left(  t\right)  =\int_{0}^{t}\left(  \frac{d\gamma^{2}}{dr}\right) 
^{1/2}\left(  t-r\right)  dW\left(  r\right)  . \label{volterra}% 
\end{equation} 
Then one can show (see \cite{MV}) that the canonical metric $\delta^{2}\left( 
s,t\right)  $ of $X$ is bounded above by $2\gamma^{2}\left(  \left\vert 
t-s\right\vert \right)  $, so that we can use the univariate $\delta 
^{2}=2\gamma^{2}$, and also $\delta^{2}\left(  s,t\right)  $ is bounded below 
by $\gamma^{2}\left(  \left\vert t-s\right\vert \right)  $. Similar 
calculations to the above then easily show that $X$ has zero $m$th variation 
as soon as $\delta^{2}\left(  r\right)  =o\left(  r^{1/\left(  2m\right) 
}\right)  $. Hence there are processes with non stationary increments that are 
more irregular than fractional Brownian for any $H>1/\left(  2m\right)  $ 
which still have zero $m$th variation: use for instance the $X$ above with 
$\gamma^{2}\left(  r\right)  =r^{1/\left(  2m\right)  }/\log\left( 
1/r\right)  $. 
\end{example} 
 
\section{Non-Gaussian case\label{NGC}} 
 
Now assume that $X$ is given by (\ref{defX}) and $M$ is a square-integrable 
(non-Gaussian) continuous martingale, $m$ is an odd integer, and define a 
positive non-random measure $\mu$ for $\bar{s}=\left(  s_{1},s_{2}% 
,\cdots,s_{m}\right)  \in\lbrack0,T]^{m}$ by% 
\begin{equation} 
\mu\left(  d\bar{s}\right)  =\mu\left(  ds_{1}ds_{2}\cdots ds_{m}\right) 
=\mathbf{E}\left[  d\left[  M\right]  \left(  s_{1}\right)  d\left[  M\right] 
\left(  s_{2}\right)  \cdots d\left[  M\right]  \left(  s_{m}\right)  \right] 
, \label{mu}% 
\end{equation} 
where $\left[  M\right]  $ is the quadratic variation process of $M$. We make 
the following assumption on $\mu$. 
 
\begin{description} 
\item[(A)] The non-negative measure $\mu$ is absolutely continuous with 
respect to the Lebesgue measure $d\bar{s}$ on $[0,T]^{m}$ and $K\left( 
\bar{s}\right)  :=d\mu/d\bar{s}$ is bounded by a tensor-power function: $0\leq 
K\left(  s_{1},s_{2},\cdots,s_{m}\right)  \leq\Gamma^{2}\left(  s_{1}\right) 
\Gamma^{2}\left(  s_{2}\right)  \cdots\Gamma^{2}\left(  s_{m}\right)  $ for 
some non-negative function $\Gamma$ on $[0,T]$. 
\end{description} 
 
A large class of processes satisfying (A) is the case where $M\left( 
t\right)  =\int_{0}^{t}H\left(  s\right)  dW\left(  s\right)  $ where $H\in 
L^{2}\left(  [0,T]\times\Omega\right)  $ and $W$ is a standard Wiener process, 
and we assume $\mathbf{E}\left[  H^{2m}\left(  t\right)  \right]  $ is finite 
for all $t\in\lbrack0,T]$. Indeed then by H\"{o}lder's inequality, since we 
can take $K\left(  \bar{s}\right)  =\mathbf{E}\left[  H^{2}\left( 
s_{1}\right)  H^{2}\left(  s_{2}\right)  \cdots H^{2}\left(  s_{m}\right) 
\right]  $, we see that $\Gamma\left(  t\right)  =$ $\left(  \mathbf{E}\left[ 
H^{2m}\left(  t\right)  \right]  \right)  ^{1/\left(  2m\right)  }$ works. 
 
We will show that the sufficient conditions for zero odd variation in the 
Gaussian cases generalize to the case of condition (A), by associating $X$ 
with the Gaussian process% 
\begin{equation} 
Z\left(  t\right)  :=\int_{0}^{T}\tilde{G}\left(  t,s\right)  dW\left( 
s\right)  . \label{Zee}% 
\end{equation} 
where $\tilde{G}\left(  t,s\right)  :=\Gamma\left(  s\right)  G\left( 
t,s\right)  $. We have the following. 
 
\begin{theorem} 
\label{MartThm}Let $m$ be an odd integer $\geq3$. Let $X$ and $Z$ be as 
defined in (\ref{defX}) and (\ref{Zee}). Assume $M$ satisfies condition 
\emph{(A)} and $Z$ is well-defined and satisfies the hypotheses of Theorem 
\ref{HomogGauss} or Theorem \ref{nonhomoggauss} relative to a univariate 
function $\delta$. Assume that for some constant $c>0$, and every small 
$\varepsilon>0$,% 
\begin{equation} 
\int_{t=2\varepsilon}^{T}dt\int_{s=0}^{t-2\varepsilon}ds\int_{u=0}% 
^{T}\left\vert \Delta\tilde{G}_{t}\left(  u\right)  \right\vert \left\vert 
\Delta\tilde{G}_{s}\left(  u\right)  \right\vert du\leq c\varepsilon\delta 
^{2}\left(  2\varepsilon\right)  , \label{additional}% 
\end{equation} 
where we use the notation $\Delta\tilde{G}_{t}\left(  u\right)  =\tilde 
{G}\left(  t+\varepsilon,u\right)  -\tilde{G}\left(  t,u\right)  $. Then $X$ 
has zero $m$th variation. 
\end{theorem} 
 
The next proposition illustrates the range of applicability of Theorem 
\ref{MartThm}. We will use it to construct classes of examples of 
martingale-based processes $X$ to which the theorem applies. 
 
\begin{proposition} 
\label{Ex}Let $X$ be defined by (\ref{defX}). Assume $m$ is an odd integer 
$\geq3$ and condition \emph{(A)} holds. Assume that $\tilde{G}\left( 
t,s\right)  :=$ $\Gamma\left(  s\right)  G\left(  t,s\right)  $ can be bounded 
above as follows: for all $s,t$,% 
\[ 
\tilde{G}\left(  t,s\right)  =\mathbf{1}_{s\leq t}\ g\left(  t,s\right) 
=\mathbf{1}_{s\leq t}\left\vert t-s\right\vert ^{1/\left(  2m\right) 
-1/2}f\left(  t,s\right) 
\] 
in which the bivariate function $f\left(  t,s\right)  $ is positive and 
bounded as% 
\[ 
\left\vert f\left(  t,s\right)  \right\vert \leq f\left(  \left\vert 
t-s\right\vert \right) 
\] 
where the univariate function $f\left(  r\right)  $ is increasing, and concave 
on $\mathbf{R}_{+}$, with $\lim_{r\rightarrow0}f\left(  r\right)  =0$, and 
where $g$ has a second mixed derivative such that% 
\begin{align*} 
\left\vert \frac{\partial g}{\partial t}\left(  t,s\right)  \right\vert 
+\left\vert \frac{\partial g}{\partial s}\left(  t,s\right)  \right\vert  & 
\leq c\left\vert t-s\right\vert ^{1/\left(  2m\right)  -3/2};\\ 
\left\vert \frac{\partial^{2}g}{\partial s\partial t}\left(  t,s\right) 
\right\vert  &  \leq c\left\vert t-s\right\vert ^{1/\left(  2m\right)  -5/2}. 
\end{align*} 
Also assume $t\mapsto g\left(  s,t\right)  $ is decreasing and $t\mapsto 
f\left(  s,t\right)  $ is increasing. Then $X$ has zero $m$-variation. 
\end{proposition} 
 
\label{pex}The presence of the indicator function $\mathbf{1}_{s\leq t}$ in 
the expression for $\tilde{G}$ above is typical of most models, since it 
coincides with asking that $Z$ be adapted to the filtrations of $W$, which is 
equivalent to $X$ being adapted to the filtration of $M$. The proposition 
covers a wide variety of martingale-based models, which can be quite far from 
Gaussian models in the sense that they can have only a few moments. We 
describe one easily constructed class. 
 
\begin{example} 
\label{ExM}Assume that $M$ is a martingale such that $\mathbf{E}\left[ 
\left\vert d\left[  M\right]  /dt\right\vert ^{m}\right]  $ is bounded above 
by a constant $c^{2m}$ uniformly in $t\leq T$. For instance we can take 
$M_{t}=\int_{0}^{t}H_{s}\left(  s\right)  dW\left(  s\right)  $ where $H$ is a 
$W$-adapted process with $\mathbf{E}\left[  \left\vert H_{s}\right\vert 
^{2m}\right]  ^{1/2m}\leq c$. This boundedness assumption implies that we can 
take $\Gamma\equiv c$ in Condition (A), and $\tilde{G}=cG$. Let $G\left( 
t,s\right)  =G_{RLfBm}\left(  t,s\right)  :=\mathbf{1}_{s\leq t}\left\vert 
t-s\right\vert ^{1/\left(  2m\right)  -1/2+\alpha}$ for some $\alpha>0$; in 
other words, $G$ is the Brownian representation kernel of the 
Riemann-Liouville fBm with parameter $H=1/\left(  2m\right)  -\alpha>1/\left( 
2m\right)  $. It is immediate to check that the assumptions of Proposition 
\ref{Ex} are satisfied for this class of martingale-based models, which 
implies that the corresponding $X$ defined by (\ref{defX}) have zero $m$th variation. 
\end{example} 
 
More generally, assume that $G$ is bounded above by a multiple of $G_{RLfBm}$, 
and assume the two partial derivatives of $G$, and the mixed second order 
derivative of $G$, are bounded by the corresponding (multiples of) derivatives 
of $G_{RLfBm}$; one can check that the standard fBm's kernel is in this class, 
and that the martingale-based models of this class also satisfy the 
assumptions of Proposition \ref{Ex}, resulting again zero $m$th variations for 
the corresponding $X$ defined in (\ref{defX}). For the sake of conciseness, we 
will omit the details, which are tedious and straightforward. 
 
The main quantitative assumption on the univariate function $\delta\left( 
\varepsilon\right)  $ corresponding to $\tilde{G}$ in the theorem, i.e. 
$\delta\left(  r\right)  =o\left(  r^{1/\left(  2m\right)  }\right)  $, can be 
reinterpreted as a regularity condition on $X$. 
 
\begin{example} 
For example, if $X$ has fractional exponential moments, in the sense that for 
some constants $a>0$ and $0<\beta\leq2$, $\mathbf{E}\left[  \exp\left( 
a\left\vert X\left(  t\right)  -X\left(  s\right)  \right\vert ^{\beta 
}\right)  \right]  $ is finite for all $s,t$, then an almost-sure uniform 
modulus of continuity for $X$ is $r\mapsto\delta\left(  r\right)  \log 
^{\beta/2}\left(  1/r\right)  $. This can be established by using Corollary 
4.5 in \cite{VV}. By using the Burkholder-Davis-Gundy inequality on the 
exponential martingale based on $M$, we can prove that such fractional 
exponential moments hold, for instance, in the setting of Example \ref{ExM}, 
if there exists $b>0$ such that $\mathbf{E}\left[  \exp\left(  b\left\vert 
H_{s}\right\vert ^{2\beta}\right)  \right]  $ is bounded in $s\in\lbrack0,T]$. 
If one only has standard (non-exponential) moments, similar (less sharp) 
results can be obtained via Kolmogorov's continuity criterion instead of 
\cite{VV}. All details are left to the reader. 
\end{example} 
 
\section{Stochastic calculus\label{STOCH}} 
 
This section's goal is to define the so-called symmetric stochastic integral 
and its associated It\^{o} formula for processes which are not fBm. The reader 
may refer to the Introduction (Section 1) for motivations on why we study this 
topic. We concentrate on Gaussian processes under hypotheses similar to those 
used in Section \ref{NonHomogGaussSect} (Theorem \ref{nonhomoggauss}). The 
basic strategy is to use the results of \cite{GNRV} which were applied to fBm. 
Let $X$ be a stochastic process on $[0,1]$. According to Sections 3 and 4 in 
\cite{GNRV} (specifically, according to the proof of part 1 of Theorem 4.4 
therein), if for every bounded measurable function $g$ on $\mathbf{R}$, the 
limit% 
\begin{equation} 
\lim_{\varepsilon\rightarrow0}\frac{1}{\varepsilon}\int_{0}^{1}du\left( 
X_{u+\varepsilon}-X_{u}\right)  ^{m}g\left(  \frac{X_{u+\varepsilon}+X_{u}}% 
{2}\right)  =0 \label{ForIto}% 
\end{equation} 
holds in probability, for both $m=3$ and $m=5$, then for every $t\in 
\lbrack0,1]$ and every $f\in C^{6}\left(  \mathbf{R}\right)  $, the 
\emph{symmetric} (\textquotedblleft generalized Stratonovich\textquotedblright% 
) stochastic integral% 
\begin{equation} 
\int_{0}^{t}f^{\prime}\left(  X_{u}\right)  d^{\circ}X_{u}=:\lim 
_{\varepsilon\rightarrow0}\frac{1}{\varepsilon}\int_{0}^{t}du\left( 
X_{u+\varepsilon}-X_{u}\right)  \frac{1}{2}\left(  f^{\prime}\left( 
X_{u+\varepsilon}\right)  +f^{\prime}\left(  X_{u}\right)  \right) 
\label{Stratoint}% 
\end{equation} 
exists and we have the It\^{o} formula% 
\begin{equation} 
f\left(  X_{t}\right)  =f\left(  X_{0}\right)  +\int_{0}^{t}f^{\prime}\left( 
X_{u}\right)  d^{\circ}X_{u}. \label{Ito}% 
\end{equation} 
Our goal is thus to prove (\ref{ForIto}) for a wide class of Gaussian 
processes $X$, which will in turn imply the existence of (\ref{Stratoint}) and 
the It\^{o} formula (\ref{Ito}). 
 
If $X$ has stationary increments in the sense of Section \ref{HomogGaussSect}, 
meaning that $\mathbf{E}\left[  \left(  X_{s}-X_{t}\right)  ^{2}\right]  $ $=$ 
$\delta^{2}\left(  t-s\right)  $ for some univariate canonical metric function 
$\delta$, then by using $g\equiv\mathbf{1}$ and our Theorem \ref{HomogGauss}, 
we see that for (\ref{ForIto}) to hold, we must have $\delta\left(  r\right) 
=o\left(  r^{1/6}\right)  $. If one wishes to treat non-stationary cases, we 
notice that (\ref{ForIto}) for $g\equiv1$ is the result of our non-stationary 
Theorem \ref{nonhomoggauss}, so it is necessary to use that theorem's 
hypotheses, which include the non-stationary version of $\delta\left( 
r\right)  =o\left(  r^{1/6}\right)  $. But we will also need some 
non-degeneracy conditions in order to apply the quartic linear regression 
method of \cite{GNRV}. These are Conditions (i) and (ii) in the next Theorem. 
Condition (iii) therein is essentially a consequence of the condition that 
$\delta^{2}$ be increasing and concave. These conditions are all further 
discussed after the statement of the next theorem and its corollary. 
 
\begin{theorem} 
\label{ForItoThm}Let $m\geq3$ be an odd integer. Let $X$ be a Gaussian process 
on $[0,1]$ satisfying the hypotheses of Theorem \ref{nonhomoggauss}. This 
means in particular that we denote as usual its canonical metric by 
$\delta^{2}\left(  s,t\right)  $, and that there exists a univariate 
increasing and concave function $\delta^{2}$ such that $\delta\left( 
r\right)  =o\left(  r^{1/(2m)}\right)  $ and $\delta^{2}\left(  s,t\right) 
\leq\delta^{2}\left(  \left\vert t-s\right\vert \right)  $. Assume that for 
$u<v$, the functions $u\mapsto Var\left[  X_{u}\right]  =:Q_{u}$, 
$v\mapsto\delta^{2}\left(  u,v\right)  $, and $u\mapsto-\delta^{2}\left( 
u,v\right)  $ are increasing and concave. Assume there exist positive 
constants $a>1$, $b<1/2$, $c>1/4$, and $c^{\prime}>0$ such that for all 
$\varepsilon<u<v\leq1$, 
 
\begin{description} 
\item[(i)] $c\delta^{2}\left(  u\right)  \leq Q_{u},$ 
 
\item[(ii)] $c^{\prime}\delta^{2}\left(  u\right)  \delta^{2}\left( 
v-u\right)  \leq Q_{u}Q_{v}-Q^{2}\left(  u,v\right)  ,$ 
 
\item[(iii)] 
\begin{equation} 
\frac{\delta\left(  au\right)  -\delta\left(  u\right)  }{\left(  a-1\right) 
u}<b\frac{\delta\left(  u\right)  }{u}. \label{concaviii}% 
\end{equation} 
 
\end{description} 
 
Then for every bounded measurable function $g$ on $\mathbf{R}$, 
\[ 
\lim_{\varepsilon\rightarrow0}\frac{1}{\varepsilon^{2}}\mathbf{E}\left[ 
\left(  \int_{0}^{1}du\left(  X_{u+\varepsilon}-X_{u}\right)  ^{m}g\left( 
\frac{X_{u+\varepsilon}+X_{u}}{2}\right)  \right)  ^{2}\right]  =0. 
\] 
 
\end{theorem} 
 
When we apply this theorem to the case $m=3$, the assumption depending on $m$, 
namely $\delta\left(  r\right)  =o\left(  r^{1/(2m)}\right)  $ is satisfied a 
fortiori for $m=5$ as well, which means that under the assumption 
$\delta\left(  r\right)  =o\left(  r^{1/6}\right)  $, the theorem's conclusion 
holds for $m=3$ and $m=5$. Therefore, as mentioned in the strategy above, we 
immediately get the following. 
 
\begin{corollary} 
\label{coroll}Assume the hypotheses of Theorem \ref{ForItoThm} with $m=3$. We 
have existence of the symmetric integral in (\ref{Stratoint}), and its It\^{o} 
formula (\ref{Ito}), for every $f\in C^{6}\left(  \mathbf{R}\right)  $ and 
$t\in\lbrack0,1]$. 
\end{corollary} 
 
The end of Section \ref{NonHomogGaussSect} contains examples satisfying the 
hypotheses of Theorem \ref{nonhomoggauss}; most of these examples also satisfy 
the monotonicity and convexity conditions in the above theorem. We state this 
formally, omitting the details of checking the conditions. 
 
\begin{example} 
The conclusion of Corollary \ref{coroll} applies to The Riemann-Liouville fBm 
described in Example \ref{RL}, which is a Gaussian process with non-stationary 
increments. It also applies to any member of the wider class of processes in 
Example \ref{RLgen} for which the function $\gamma$ defined therein satisfies 
conditions (i), (ii), and (iii) of Theorem \ref{ForItoThm}. This includes the 
family of processes such that $\gamma\left(  r\right)  =r^{H}\log^{\beta 
}\left(  1/r\right)  $ for $H\in(1/6,1)$ and $\beta\in\mathbf{R}$, the case 
$\beta=0$ yielding the Riemann-Liouville fBm processes. 
\end{example} 
 
Before proceeding to the proof of Theorem \ref{ForItoThm}, we provide a 
broader discussion of its hypotheses. 
 
Condition (i) is a type of coercivity assumption on the non-degeneracy of 
$X$'s variances in comparison to its increments' variances. The hypotheses of 
Theorem \ref{nonhomoggauss} imply that $Q_{u}\leq\delta^{2}\left(  u\right) 
$, and Condition (i) simply adds that these two quantities should be 
commensurate, with a lower bound that it not too small. The "Volterra 
convolution"-type class of processes (\ref{volterra}) given at the end of 
Section \ref{NonHomogGaussSect}, which includes the Riemann-Liouville fBm's, 
satisfies Condition (i) with $c=1/2$. In the stationary case, (i) is trivially 
satisfied since $Q_{u}\equiv\delta^{2}\left(  u\right)  $. 
 
Condition (ii) is also a type of coercivity condition. It too is satisfied in 
the stationary case. We prove this claim, since it is not immediately obvious. 
In the stationary case, since $\delta^{2}\left(  u,v\right)  =\delta 
^{2}\left(  v-u\right)  =Q_{v-u}$, we calculate% 
\[ 
Q_{u}Q_{v}-Q^{2}\left(  u,v\right)  =Q_{u}Q_{v}-4^{-1}\left(  Q_{u}% 
+Q_{v}-Q_{v-u}\right)  ^{2}% 
\] 
and after rearranging some terms we obtain% 
\[ 
Q_{u}Q_{v}-Q^{2}\left(  u,v\right)  =2^{-1}Q_{v-u}\left(  Q_{u}+Q_{v}\right) 
-4^{-1}\left(  Q_{v}-Q_{u}\right)  ^{2}-4^{-1}Q_{v-u}^{2}. 
\] 
We note first that by the concavity of $Q$, we have $Q_{v}-Q_{u}<Q_{v-u}$, and 
consequently, $\left(  Q_{v}-Q_{u}\right)  ^{2}\leq\left(  Q_{v}-Q_{u}\right) 
Q_{v-u}\leq Q_{v}Q_{v-u}$. This implies% 
\[ 
Q_{u}Q_{v}-Q^{2}\left(  u,v\right)  \geq2^{-1}Q_{v-u}Q_{u}+4^{-1}\left( 
Q_{v-u}Q_{v}-Q_{v-u}^{2}\right)  . 
\] 
Now by monotonicity of $Q$, we can write $Q_{v-u}Q_{v}\geq Q_{v-u}^{2}$. This, 
together with Condition (i), yield Condition (ii) since we now have% 
\[ 
Q_{u}Q_{v}-Q^{2}\left(  u,v\right)  \geq2^{-1}Q_{v-u}Q_{u}\geq2^{-1}% 
c^{2}\delta^{2}\left(  v-u\right)  \delta^{2}\left(  u\right)  . 
\]

Lastly, Condition (iii) represents a strengthened concavity condition on the 
univariate function $\delta$. Indeed, the left-hand side in (\ref{concaviii}) 
is the slope of the secant of the graph of $\delta$ between the points $u$ and 
$au$, while the right-hand side is $b$ times the slope of the secant from $0$ 
to $u$. If $b$ were allowed to be $1$, (iii) would simply be a consequence of 
convexity. Here taking $b\leq1/2$ means that we are exploiting the concavity 
of $\delta^{2}$; the fact that condition (iii) requires slightly more, namely 
$b$ strictly less than $1/2$, allows us to work similarly to the scale 
$\delta\left(  r\right)  =r^{H}$ with $H<1/2$, as opposed to simply asking 
$H\leq1/2$. Since the point of the Theorem is to allow continuity moduli which 
are arbitrarily close to $r^{1/6}$, Condition (iii) is hardly a 
restriction.\bigskip 
 
\textbf{Proof of Theorem \ref{ForItoThm}.} 
 
\bigskip 
 
\noindent\emph{Step 0: setup.} The expectation to be evaluated is written, as 
usual, as a double integral over $\left(  u,v\right)  \in\lbrack0,1]^{2}$. For 
$\varepsilon>0$ fixed, we define the \textquotedblleft 
off-diagonal\textquotedblright\ set 
\[ 
D_{\varepsilon}=\left\{  \left(  u,v\right)  \in\lbrack0,1]^{2}:\varepsilon 
^{1-\rho}\leq u\leq v-\varepsilon^{1-\rho}<v\leq1\right\} 
\] 
where $\rho\in(0,1)$ is fixed. Using the boundedness of $g$ and 
Cauchy-Schwarz's inequality, thanks to the hypothesis $\delta\left(  r\right) 
=o\left(  r^{1/\left(  2m\right)  }\right)  $, the term corresponding to the 
diagonal part (integral over $D_{\varepsilon}^{c}$) can be treated identically 
to what was done in \cite{GNRV} in dealing with their term $\mathcal{J}% 
^{\prime}\left(  \varepsilon\right)  $ following the statement of their Lemma 
5.1, by choosing $\rho$ small enough. It is thus sufficient to prove that 
\[ 
\mathcal{J}\left(  \varepsilon\right)  :=\frac{1}{\varepsilon^{2}}% 
\mathbf{E}\left[  \iint_{D_{\varepsilon}}dudv\left(  X_{u+\varepsilon}% 
-X_{u}\right)  ^{m}\left(  X_{v+\varepsilon}-X_{v}\right)  ^{m}g\left( 
\frac{X_{u+\varepsilon}+X_{u}}{2}\right)  g\left(  \frac{X_{v+\varepsilon 
}+X_{v}}{2}\right)  \right] 
\] 
tends to $0$ as $\varepsilon$ tends to $0$. We now use the same method and 
notation as in Step 3 of the proof of Theorem 4.1 in \cite{GNRV}. It proceeds 
through the linear regression analysis of the Gaussian vector $\left( 
G_{1,}G_{2},G_{3},G_{4}\right)  :=(X_{u+\varepsilon}+X_{u},X_{v+\varepsilon 
}+X_{v},X_{u+\varepsilon}-X_{u},X_{v+\varepsilon}-X_{v})$. In order to avoid 
repeating arguments from that proof, we only state and prove the new lemmas 
which are required. The new elements come from the analysis of the Gaussian 
vector $\left(  \Gamma_{3},\Gamma_{4}\right)  ^{t}:=A\left(  G_{1}% 
,G_{2}\right)  $ where $A:=\Lambda_{21}\left(  \Lambda_{11}\right)  ^{-1}$ 
where $\Lambda_{11}$ is the covariance of the vector $\left(  G_{1,}% 
G_{2}\right)  $ and $\Lambda_{21}$ is the matrix $\left\{  Cov\left( 
G_{i+2},G_{j}\right)  :i,j=1,2\right\}  $, as well as from the the centered 
Gaussian vector $\left(  Z_{3},Z_{4}\right)  $ which is the component 
independent of $\left(  G_{3},G_{4}\right)  $ in its linear regression against 
$\left(  G_{1,}G_{2}\right)  $, i.e. $\left(  G_{3},G_{4}\right) 
^{t}=A\left(  G_{1,}G_{2}\right)  ^{t}+\left(  Z_{3},Z_{4}\right)  $% 
.\vspace{0.15in} 
 
\noindent\emph{Step 1: translating Lemma 5.3 from \cite{GNRV}.} Using the fact 
that $\mathbf{E}\left[  Z_{\ell}^{2}\right]  \leq\mathbf{E}\left[  G_{\ell 
}^{2}\right]  \leq\delta^{2}\left(  \varepsilon\right)  $, this lemma 
translates as the following, proved in the Appendix: 
 
\begin{lemma} 
\label{Lemma53}Let $k\geq2$ be an integer. Then for $\ell=3,4$, 
\[ 
\iint_{D_{\varepsilon}}\mathbf{E}\left[  \left\vert \Gamma_{\ell}\right\vert 
^{k}\right]  dudv\leq cst\cdot\varepsilon\delta^{k}\left(  \varepsilon\right) 
. 
\] 
 
\end{lemma} 
 
\noindent\emph{Step 2: translating Lemma 5.4 from \cite{GNRV}.} We will prove 
the following result 
 
\begin{lemma} 
\label{Lemma54}For all $j\in\{0,1,\cdots,\left(  m-1\right)  /2\}$,% 
\[ 
\iint_{D_{\varepsilon}}\left\vert \mathbf{E}\left[  Z_{3}Z_{4}\right] 
\right\vert ^{m-2j}dudv\leq cst\cdot\varepsilon\delta^{2\left(  m-2j\right) 
}\left(  \varepsilon\right)  . 
\] 
 
\end{lemma} 
 
\textbf{Proof of Lemma \ref{Lemma54}.} As in \cite{GNRV}, we have% 
\[ 
\left\vert \mathbf{E}\left[  Z_{3}Z_{4}\right]  \right\vert ^{m-2j}\leq 
cst\cdot\left\vert \mathbf{E}\left[  G_{3}G_{4}\right]  \right\vert 
^{m-2j}+cst\cdot\left\vert \mathbf{E}\left[  \Gamma_{3}\Gamma_{4}\right] 
\right\vert ^{m-2j}. 
\] 
The required estimate for the term corresponding to $\left\vert \mathbf{E}% 
\left[  \Gamma_{3}\Gamma_{4}\right]  \right\vert ^{m-2j}$ follows by 
Cauchy-Schwarz's inequality and Lemma \ref{Lemma53}. For the term 
corresponding to $\left\vert \mathbf{E}\left[  G_{3}G_{4}\right]  \right\vert 
^{m-2j}$, we recognize that $\mathbf{E}\left[  G_{3}G_{4}\right]  $ is the 
negative planar increment $\Theta^{\varepsilon}\left(  u,v\right)  $ defined 
in (\ref{defDelta}). Thus the corresponding term was already considered in the 
proof of Theorem (\ref{nonhomoggauss}). More specifically, up to the factor 
$\varepsilon^{2}\delta^{-4j}\left(  \varepsilon\right)  $, we now have to 
estimated the same integral as in Step 2 of that theorem's proof: see 
expression (\ref{JjOD}) for the term we called $J_{j,OD}$. This means that% 
\[ 
\iint_{D_{\varepsilon}}\left\vert \mathbf{E}\left[  G_{3}G_{4}\right] 
\right\vert ^{m-2j}dudv\leq\frac{\varepsilon^{2}}{\delta^{4j}\left( 
\varepsilon\right)  }J_{j,OD}\leq\varepsilon^{2}\left\vert \mu\right\vert 
\left(  OD\right)  \delta^{2\left(  m-2j-1\right)  }\left(  \varepsilon 
\right)  . 
\] 
Our hypotheses borrowed from Theorem (\ref{nonhomoggauss}) that $\left\vert 
\mu\right\vert \left(  OD\right)  \leq cst\cdot\varepsilon^{1/m-1}$ and that 
$\delta^{2}\left(  \varepsilon\right)  =o\left(  r^{1/\left(  2m\right) 
}\right)  $ now imply that the above is $\ll\varepsilon\delta^{2\left( 
m-2j\right)  }\left(  \varepsilon\right)  $, concluding the lemma's 
proof.\hfill$\square$\vspace{0.15in} 
 
\noindent\emph{Step 4. Conclusion}. The remainder of the proof of the theorem 
is to check that Lemmas \ref{Lemma53} and \ref{Lemma54} do imply the claim of 
the theorem; this is done exactly as in Steps 3 and 4 of the proof of Theorem 
4.1 in \cite{GNRV}. Since such a task is only bookkeeping, we omit it, 
concluding the proof of Theorem \ref{ForItoThm}, modulo the proof of Lemma 
\ref{Lemma53} which is found in the appendix.\hfill$\blacksquare$ 
 
\begin{center} 
\textbf{Acknowledgements} 
\end{center} 
 
\noindent The work of F. Russo was partially supported by the ANR Project 
MASTERIE 2010 BLAN-0121-01. The work of F. Viens is partially supported by NSF 
DMS grant 0907321. Constructive comments by referees and editors are 
gratefully acknowledged and resulted in several improvements.

\section{Appendix\label{APP}} 
 
\begin{proof} 
[Proof of Lemma \ref{lemma1}]The formula in the lemma is an easy consequence 
of the following formula, which can be found as Lemma 5.2 in \cite{GNRV}: for 
any centered jointly Gaussian pair of r.v.'s $\left(  Y,Z\right)  $, we have 
$\mathbf{E}\left[  Y^{m}Z^{m}\right]  =\sum_{j=0}^{\left(  m-1\right) 
/2}c_{j}\mathbf{E}\left[  YZ\right]  ^{m-2j}\ Var\left[  X\right]  ^{j}% 
$\ $Var\left[  Y\right]  ^{j}.$ To see that the $c_{j}$'s are positive, note 
that one can decompose each odd monomial into the basis of odd-order Hermite 
polynomials: $x^{m}=\sum_{j=0}^{\left(  m-1\right)  /2}a_{2j+1}H_{2j+1}\left( 
x\right)  $, from whence it follows, thanks to the orthogonality of Hermite 
polynomials of Gaussian rv's, that $c_{j}=\left(  a_{2j+1}\right)  ^{2}% 
$.$\vspace*{0.1in}$ 
\end{proof} 
 
\begin{proof} 
[\noindent Proof of Lemma \ref{X3eExact}]The proof of this lemma is 
elementary. It follows from two uses of the multiplication formula for Wiener 
integrals \cite[Proposition 1.1.3]{Nbook}, for instance. All details are left 
to the reader.$\vspace*{0.1in}$ 
\end{proof} 
 
\begin{proof} 
[Proof of Lemma \ref{I1pro}]Reintroducing the notation $X$ and $\Theta$ into 
the formula in Lemma \ref{X3eExact}, we get% 
\[ 
\mathcal{I}_{1}=\frac{3}{\varepsilon}\int_{0}^{T}ds\left(  X\left( 
s+\varepsilon\right)  -X\left(  s\right)  \right)  Var\left(  X\left( 
s+\varepsilon\right)  -X\left(  s\right)  \right) 
\] 
and therefore, 
\[ 
\mathbf{E}\left[  \left\vert \mathcal{I}_{1}\right\vert ^{2}\right]  =\frac 
{9}{\varepsilon^{2}}\int_{0}^{T}\int_{0}^{t}dtds\Theta^{\varepsilon}\left( 
s,t\right)  Var\left(  X\left(  t+\varepsilon\right)  -X\left(  t\right) 
\right)  Var\left(  X\left(  s+\varepsilon\right)  -X\left(  s\right) 
\right) 
\] 
Using the variances of fBm, writing $H$ instead of $1/6$ to improve 
readability,% 
\begin{align*} 
\mathbf{E}\left[  \left\vert \mathcal{I}_{1}\right\vert ^{2}\right]   & 
=\frac{9}{2}\varepsilon^{-2+4H}\int_{0}^{T}\int_{0}^{T}dtds~Cov\left[ 
X\left(  t+\varepsilon\right)  -X\left(  t\right)  ;X\left(  s+\varepsilon 
\right)  -X\left(  s\right)  \right] \\ 
&  =\frac{9}{2}\varepsilon^{-2+4H}~Var\left[  \int_{0}^{T}\left(  X\left( 
t+\varepsilon\right)  -X\left(  t\right)  \right)  dt\right] \\ 
&  =\frac{9}{2}\varepsilon^{-2+4H}~Var\left[  \int_{T}^{T+\varepsilon}X\left( 
t\right)  dt-\int_{0}^{\varepsilon}X\left(  t\right)  dt\right]  . 
\end{align*} 
Bounding the variance of the difference by twice the sum of the variances, 
\[ 
\mathbf{E}\left[  \left\vert \mathcal{I}_{1}\right\vert ^{2}\right] 
\leq9\varepsilon^{-2+4H}\left(  \int_{T}^{T+\varepsilon}\int_{T}% 
^{T+\varepsilon}T^{2H}dsdt+\int_{0}^{\varepsilon}\int_{0}^{\varepsilon 
}\varepsilon^{2H}dsdt\right)  =O\left(  \varepsilon^{4H}\right)  , 
\] 
proving Lemma \ref{I1pro}.$\vspace*{0.1in}$ 
\end{proof} 
 
\begin{proof} 
[Proof of Lemma \ref{I3pro}]By the technique at the start of the proof of 
Lemma \ref{I1pro}, the product formula in \cite[Proposition 1.1.3]{Nbook}, and 
the covariance of fBm, we first get% 
 
\[ 
\mathcal{I}_{3}:=\frac{6}{\varepsilon}\int_{0}^{T}dW\left(  s_{3}\right) 
\int_{0}^{s_{3}}dW\left(  s_{2}\right)  \int_{0}^{s_{2}}dW\left( 
s_{1}\right)  \int_{0}^{T}\left[  \prod_{k=1}^{3}\Delta G_{s}\left( 
s_{k}\right)  \right]  ds. 
\]% 
\begin{align*} 
\mathbf{E}\left[  \left\vert \mathcal{I}_{3}\right\vert ^{2}\right]   & 
=\frac{12}{\varepsilon^{2}}\int_{0}^{T}\int_{0}^{t}dtds~\left(  \Theta 
^{\varepsilon}\left(  s,t\right)  \right)  ^{3}\\ 
&  =\frac{6}{\varepsilon^{2}}\int_{0}^{T}\int_{0}^{t}dtds~\left(  \left\vert 
t-s+\varepsilon\right\vert ^{2H}+\left\vert t-s-\varepsilon\right\vert 
^{2H}-2\left\vert t-s\right\vert ^{2H}\right)  ^{3}. 
\end{align*} 
We must take care of the absolute values, i.e. of whether $\varepsilon$ is 
greater or less than $t-s$. We define the \textquotedblleft 
off-diagonal\textquotedblright\ portion of $\mathbf{E}\left[  \left\vert 
\mathcal{I}_{3}\right\vert ^{2}\right]  $ as 
\[ 
\mathcal{ODI}_{3}:=6\varepsilon^{-2}\int_{2\varepsilon}^{T}\int_{0}% 
^{t-2\varepsilon}dtds\left(  \left\vert t-s+\varepsilon\right\vert 
^{2H}+\left\vert t-s-\varepsilon\right\vert ^{2H}-2\left\vert t-s\right\vert 
^{2H}\right)  ^{3}. 
\] 
For $s,t$ in the integration domain for the above integral, since $\bar 
{t}:=t-s>2\varepsilon$, by two iterated applications of the Mean Value Theorem 
for the function $x^{2H}$ on the intervals $[\bar{t}-\varepsilon,\bar{t}]$ and 
$[\bar{t},\bar{t}+\varepsilon]$,% 
\[ 
\left\vert \bar{t}+\varepsilon\right\vert ^{2H}+\left\vert \bar{t}% 
-\varepsilon\right\vert ^{2H}-2\bar{t}^{2H}=2H\left(  2H-1\right) 
\varepsilon\left(  \xi_{1}-\xi_{2}\right)  \xi^{2H-2}% 
\] 
for some $\xi_{2}\in\lbrack\bar{t}-\varepsilon,\bar{t}],$ $\xi_{1}\in 
\lbrack\bar{t},\bar{t}+\varepsilon]$, and $\xi\in\lbrack\xi_{1},\xi_{2}]$, and 
therefore 
\begin{align*} 
\left\vert \mathcal{ODI}_{3}\right\vert  &  \leq384H^{3}\left\vert 
2H-1\right\vert ^{3}\varepsilon^{-2}\int_{2\varepsilon}^{T}\int_{0}% 
^{t-2\varepsilon}\left(  \varepsilon\cdot2\varepsilon\cdot\left( 
t-s-\varepsilon\right)  ^{2H-2}\right)  ^{3}dtds\\ 
&  \leq\frac{384H^{3}\left\vert 2H-1\right\vert ^{3}}{5-6H}T\varepsilon 
^{6H-1}=\frac{32}{243}T. 
\end{align*} 
where in the last line we substituted $H=1/6$. Thus the \textquotedblleft 
off-diagonal\textquotedblright\ term is bounded. The diagonal part of 
$\mathcal{I}_{3}$ is% 
\begin{align*} 
\mathcal{DI}_{3}  &  :=6\varepsilon^{-2}\int_{0}^{T}\int_{t-2\varepsilon}% 
^{t}dtds\left(  \left\vert t-s+\varepsilon\right\vert ^{2H}+\left\vert 
t-s-\varepsilon\right\vert ^{2H}-2\left\vert t-s\right\vert ^{2H}\right) 
^{3}\\ 
&  =6\varepsilon^{-1+6H}T\int_{0}^{2}dr\left(  \left\vert r+1\right\vert 
^{2H}+\left\vert r-1\right\vert ^{2H}-2\left\vert r\right\vert ^{2H}\right) 
^{3}dr=CT 
\end{align*} 
where, having substituted $H=1/6$, yields that $C$ is a universal constant. 
Thus the diagonal part $\mathcal{DI}_{3}$ of $\mathbf{E}[\mathbf{|}% 
\mathcal{I}_{3}|^{2}]$ is constant. This proves that $\mathcal{I}_{3}$ is 
bounded in $L^{2}\left(  \Omega\right)  $, as announced. To conclude that it 
cannot converge in $L^{2}\left(  \Omega\right)  $, recall that from 
\cite[Theorem 4.1 part (2)]{GNRV}, $[X,3]_{\varepsilon}\left(  T\right) 
=\mathcal{I}_{1}+\mathcal{I}_{3}$ converges in distribution to a 
non-degenerate normal law. By Lemma \ref{I1pro}, $\mathcal{I}_{1}$ converges 
to $0$ in $L^{2}\left(  \Omega\right)  $. Therefore, $\mathcal{I}_{3}$ 
converges in distribution to a non-degenerate normal law; if it also converged 
in $L^{2}\left(  \Omega\right)  $, since the 3rd Wiener chaos is closed in 
$L^{2}\left(  \Omega\right)  $, the limit would have to be in that same chaos, 
and thus would not have a non-degenerate normal law.$\vspace*{0.1in}$ 
\end{proof} 
 
\begin{proof} 
[Proof of Theorem \ref{nonhomoggauss}] 
 
\noindent\emph{Step 0: setup.} Recall the result of Lemma \ref{lemma1}, where 
now we express $Var\left[  X\left(  t+\varepsilon\right)  -X\left(  t\right) 
\right]  =\delta^{2}\left(  t,t+\varepsilon\right)  $ and 
\begin{equation} 
\Theta^{\varepsilon}\left(  s,t\right)  =\mu\left(  \lbrack s,s+\varepsilon 
]\times\lbrack t,t+\varepsilon)\right)  =\int_{s}^{s+\varepsilon}\int 
_{t}^{t+\varepsilon}\mu\left(  dudv\right)  . \label{Thetamubirdie}% 
\end{equation} 
We again separate the diagonal term from the off-diagonal term, although this 
time the diagonal is twice as wide: it is defined as $\{\left(  s,t\right) 
:0\leq t-2\varepsilon\leq s\leq t\}$.\vspace{0.1in} 
 
\noindent\emph{Step 1: diagonal.} Using Cauchy-Schwarz's inequality which 
implies $\left\vert \Theta^{\varepsilon}\left(  s,t\right)  \right\vert 
\leq\delta\left(  s,s+\varepsilon\right)  \delta\left(  t,t+\varepsilon 
\right)  $, and bounding each term $\delta\left(  s,s+\varepsilon\right)  $ by 
$\delta\left(  \varepsilon\right)  $, the diagonal portion of $\mathbf{E}% 
\left[  \left(  [X,m]_{\varepsilon}\left(  T\right)  \right)  ^{2}\right]  $ 
can be bounded above, in absolute value, by% 
\[ 
\frac{1}{\varepsilon^{2}}\sum_{j=0}^{\left(  m-1\right)  /2}c_{j}% 
\int_{2\varepsilon}^{T}dt\int_{t-2\varepsilon}^{t}ds\delta^{2m}\left( 
\varepsilon\right)  =cst\cdot\varepsilon^{-1}\delta^{2m}\left(  \varepsilon 
\right)  . 
\] 
Hypothesis (\ref{nhdeltacond}) implies that this converges to $0$ with 
$\varepsilon$. The case of $t\leq2\varepsilon$ works equally easily.\vspace 
{0.1in} 
 
\noindent\emph{Step 2: off diagonal.} The off-diagonal contribution is the sum 
for $j=0,\cdots,\left(  m-1\right)  /2$ of the terms% 
\begin{equation} 
J_{j,OD}=\varepsilon^{-2}c_{j}\int_{2\varepsilon}^{T}dt\int_{0}% 
^{t-2\varepsilon}ds\delta^{2j}\left(  s,s+\varepsilon\right)  \delta 
^{2j}\left(  t,t+\varepsilon\right)  \Theta^{\varepsilon}\left(  s,t\right) 
^{m-2j} \label{JjOD}% 
\end{equation}

\noindent\emph{Step 2.1: term }$J_{\left(  m-1\right)  /2,OD}$. This is the 
dominant term. Denoting $c=\left\vert c_{\left(  m-1\right)  /2}\right\vert $, 
we have% 
\[ 
\left\vert J_{\left(  m-1\right)  /2,OD}\right\vert \leq\frac{c\delta 
^{2m-2}\left(  \varepsilon\right)  }{\varepsilon^{2}}\int_{2\varepsilon}% 
^{T}dt\int_{0}^{t-2\varepsilon}ds\left\vert \Theta^{\varepsilon}\left( 
s,t\right)  \right\vert . 
\] 
We estimate the integral, using the formula (\ref{Thetamubirdie}) and Fubini's 
theorem:% 
\begin{align*} 
&  \int_{2\varepsilon}^{T}dt\int_{0}^{t-2\varepsilon}ds\left\vert 
\Theta^{\varepsilon}\left(  s,t\right)  \right\vert =\int_{2\varepsilon}% 
^{T}dt\int_{0}^{t-2\varepsilon}ds\left\vert \int_{s}^{s+\varepsilon}\int 
_{t}^{t+\varepsilon}\mu\left(  dudv\right)  \right\vert \\ 
&  \leq\int_{2\varepsilon}^{T}dt\int_{0}^{t-2\varepsilon}ds\int_{s}% 
^{s+\varepsilon}\int_{t}^{t+\varepsilon}\left\vert \mu\right\vert \left( 
dudv\right)  =\int_{2\varepsilon}^{T+\varepsilon}\int_{0}^{v\wedge 
(T-\varepsilon)}\left\vert \mu\right\vert \left(  dudv\right)  \int 
_{2\varepsilon\vee\left(  v-\varepsilon\right)  \vee(u+\varepsilon)}^{v\wedge 
T}\int_{0\vee(u-\varepsilon)}^{u\wedge\left(  t-2\varepsilon\right)  }ds\ dt\\ 
&  \leq\int_{2\varepsilon}^{T+\varepsilon}\int_{0}^{v-\varepsilon}\left\vert 
\mu\right\vert \left(  dudv\right)  \int_{v-\varepsilon}^{v}\int 
_{u-\varepsilon}^{u}ds\ dt=\varepsilon^{2}\int_{2\varepsilon}^{T+\varepsilon 
}\int_{0}^{v-\varepsilon}\left\vert \mu\right\vert \left(  dudv\right)  . 
\end{align*} 
Hence we have% 
\[ 
J_{\left(  m-1\right)  /2,OD}\leq c\delta^{2m-2}\left(  \varepsilon\right) 
\int_{v=2\varepsilon}^{T+\varepsilon}\int_{u=0}^{v-\varepsilon}\left\vert 
\mu\right\vert \left(  dudv\right)  \leq c\delta^{2m-2}\left(  \varepsilon 
\right)  \left\vert \mu\right\vert \left(  OD\right)  , 
\] 
which again converges to $0$ by hypothesis as $\varepsilon$ goes to 
$0$.\vspace{0.1in} 
 
\noindent\emph{Step 2.2: other }$J_{j,OD}$ \emph{terms}. Let now $j<\left( 
m-1\right)  /2$. Using Cauchy-Schwarz's inequality for all but one of the 
$m-2j$ factors $\Theta$ in the expression (\ref{JjOD}) for $J_{j,OD}$, which 
is allowed because $m-2j\geq1$ here, exploiting the bounds on the variance 
terms via the univariate function $\delta$, we have% 
\[ 
\left\vert J_{j,OD}\right\vert \leq\delta^{2m-2}\left(  \varepsilon\right) 
c_{j}\varepsilon^{-2}\int_{2\varepsilon}^{T}dt\int_{0}^{t-2\varepsilon 
}ds\left\vert \Theta^{\varepsilon}\left(  s,t\right)  \right\vert , 
\] 
which is the same term we estimated in Step 2.1. This finishes the proof of 
the theorem.$\vspace*{0.1in}$ 
\end{proof} 
 
\begin{proof} 
[Proof of Theorem \ref{MartThm}]\noindent\emph{Step 0: setup.} We use an 
expansion for powers of martingales written explicitly at Corollary 2.18 of 
\cite{RE}. For any integer $k\in\lbrack0,\left[  m/2\right]  ]$, let 
$\Sigma_{m}^{k}$ be the set of permutations $\sigma$ of $m-k$ defined as those 
for which the first $k$ terms $\sigma^{-1}\left(  1\right)  ,\sigma 
^{-1}\left(  2\right)  ,\cdots,\sigma^{-1}\left(  k\right)  $ are chosen 
arbitrarily and the next $m-2k$ terms are chosen arbitrarily among the 
remaining integers $\left\{  1,2,\cdots,m-k\right\}  \setminus\left\{ 
\sigma^{-1}\left(  1\right)  ,\sigma^{-1}\left(  2\right)  ,\cdots,\sigma 
^{-1}\left(  k\right)  \right\}  $. Let $Y$ be a fixed square-integrable 
martingale. We define the process $Y_{\sigma,\ell}$ (denoted in the above 
reference by $\sigma_{Y}^{\ell}$) by setting, for each $\sigma\in\Sigma 
_{m}^{k}$ and each $\ell=1,2,\cdots,m-k$,% 
\[ 
Y_{\sigma,\ell}\left(  t\right)  =\left\{ 
\begin{array} 
[c]{c}% 
\left[  Y\right]  \left(  t\right)  \ \mbox{ if }\ \sigma\left(  \ell\right) 
\in\left\{  1,2,\cdots,k\right\} \\ 
Y\left(  t\right)  \ \mbox{ if }\ \sigma\left(  \ell\right)  \in\left\{ 
k+1,\cdots,m-k\right\}  . 
\end{array} 
\right. 
\] 
From Corollary 2.18 of \cite{RE}, we then have for all $t\in\lbrack0,T]$% 
\[ 
\left(  Y_{t}\right)  ^{m}=\sum_{k=0}^{[m/2]}\frac{m!}{2^{k}}\sum_{\sigma 
\in\Sigma_{m}^{k}}\int_{0}^{t}\int_{0}^{u_{m-k}}\cdots\int_{0}^{u_{2}% 
}dY_{\sigma,1}\left(  u_{1}\right)  \ dY_{\sigma,2}\left(  u_{2}\right) 
\cdots dY_{\sigma,m-k}\left(  u_{m-k}\right)  . 
\]

We use this formula to evaluate% 
\[ 
\lbrack X,m]_{\varepsilon}\left(  T\right)  =\frac{1}{\varepsilon}\int_{0}% 
^{T}ds\left(  X\left(  s+\varepsilon\right)  -X\left(  s\right)  \right)  ^{m}% 
\] 
by noting that the increment $X\left(  s+\varepsilon\right)  -X\left( 
s\right)  $ is the value at time $T$ of the martingale $Y_{t}:=\int_{0}% 
^{t}\Delta G_{s}\left(  u\right)  dM\left(  u\right)  $ where we set% 
\[ 
\Delta G_{s}\left(  u\right)  :=G\left(  s+\varepsilon,u\right)  -G\left( 
s,u\right)  . 
\] 
Hence% 
\begin{align*} 
&  \left(  X\left(  s+\varepsilon\right)  -X\left(  s\right)  \right)  ^{m}\\ 
&  =\sum_{k=0}^{[m/2]}\frac{m!}{2^{k}}\sum_{\sigma\in\Sigma_{m}^{k}}\int 
_{0}^{T}\int_{0}^{u_{m-k}}\cdots\int_{0}^{u_{2}}d\left[  M\right]  \left( 
u_{\sigma\left(  1\right)  }\right)  \left\vert \Delta G_{s}\left( 
u_{\sigma\left(  1\right)  }\right)  \right\vert ^{2}\cdots d\left[  M\right] 
\left(  u_{\sigma\left(  k\right)  }\right)  \left\vert \Delta G_{s}\left( 
u_{\sigma\left(  k\right)  }\right)  \right\vert ^{2}\\ 
&  dM\left(  u_{\sigma\left(  k+1\right)  }\right)  \Delta G_{s}\left( 
u_{\sigma\left(  k+1\right)  }\right)  \cdots dM\left(  u_{\sigma\left( 
m-k\right)  }\right)  \Delta G_{s}\left(  u_{\sigma\left(  m-k\right) 
}\right)  . 
\end{align*} 
Therefore we can write% 
\begin{align*} 
&  [X,m]_{\varepsilon}\left(  T\right) \\ 
&  =\frac{1}{\varepsilon}\sum_{k=0}^{[m/2]}\frac{m!}{2^{k}}\sum_{\sigma 
\in\Sigma_{m}^{k}}\int_{0}^{T}\int_{0}^{u_{m-k}}\cdots\int_{0}^{u_{2}}d\left[ 
M\right]  \left(  u_{\sigma\left(  1\right)  }\right)  \cdots d\left[ 
M\right]  \left(  u_{\sigma\left(  k\right)  }\right)  dM\left( 
u_{\sigma\left(  k+1\right)  }\right)  \cdots dM\left(  u_{\sigma\left( 
m-k\right)  }\right) \\ 
&  \left[  \Delta G_{\cdot}\left(  u_{\sigma\left(  k+1\right)  }\right) 
;\cdots;\Delta G_{\cdot}\left(  u_{\sigma\left(  m-k\right)  }\right)  ;\Delta 
G_{\cdot}\left(  u_{\sigma\left(  1\right)  }\right)  ;\Delta G_{\cdot}\left( 
u_{\sigma\left(  1\right)  }\right)  ;\cdots;\Delta G_{\cdot}\left( 
u_{\sigma\left(  k\right)  }\right)  ;\Delta G_{\cdot}\left(  u_{\sigma\left( 
k\right)  }\right)  \right]  , 
\end{align*} 
where we have used the notation 
\[ 
\left[  f_{1},f_{2},\cdots,f_{m}\right]  :=\int_{0}^{T}f_{1}\left(  s\right) 
f_{2}\left(  s\right)  \cdots f_{m}\left(  s\right)  ds. 
\] 
To calculate the expected square of the above, we will bound it above by the 
sum over $k$ and $\sigma$ of the expected square of each term. Writing squares 
of Lebesgue integrals as double integrals, and using It\^{o}'s formula, each 
term's expected square is thus, up to $\left(  m,k\right)  $-dependent 
multiplicative constants, equal to the expression% 
\begin{align} 
K  &  =\frac{1}{\varepsilon^{2}}\int_{u_{m-k}=0}^{T}\int_{u_{m-k}^{\prime}% 
=0}^{T}\int_{u_{m-k-1}=0}^{u_{m-k}}\int_{u_{m-k-1}^{\prime}=0}^{u_{m-k}}% 
\cdots\int_{u_{1}=0}^{u_{2}}\int_{u_{1}^{\prime}=0}^{u_{2}}\nonumber\\ 
&  \mathbf{E}\left[  d\left[  M\right]  ^{\otimes k}\left(  u_{\sigma\left( 
1\right)  },\cdots,u_{\sigma\left(  k\right)  }\right)  d\left[  M\right] 
^{\otimes k}\left(  u_{\sigma\left(  1\right)  }^{\prime},\cdots 
,u_{\sigma\left(  k\right)  }^{\prime}\right)  d\left[  M\right] 
^{\otimes\left(  m-2k\right)  }\left(  u_{\sigma\left(  k+1\right)  }% 
,\cdots,u_{\sigma\left(  m-k\right)  }\right)  \right] \nonumber\\ 
&  \cdot\left[  \Delta G_{\cdot}\left(  u_{\sigma\left(  k+1\right)  }\right) 
;\cdots;\Delta G_{\cdot}\left(  u_{\sigma\left(  m-k\right)  }\right)  ;\Delta 
G_{\cdot}\left(  u_{\sigma\left(  1\right)  }\right)  ;\Delta G_{\cdot}\left( 
u_{\sigma\left(  1\right)  }\right)  ;\cdots;\Delta G_{\cdot}\left( 
u_{\sigma\left(  k\right)  }\right)  ;\Delta G_{\cdot}\left(  u_{\sigma\left( 
k\right)  }\right)  \right] \nonumber\\ 
&  \cdot\left[  \Delta G_{\cdot}\left(  u_{\sigma\left(  k+1\right)  }\right) 
;\cdots;\Delta G_{\cdot}\left(  u_{\sigma\left(  m-k\right)  }\right)  ;\Delta 
G_{\cdot}\left(  u_{\sigma\left(  1\right)  }^{\prime}\right)  ;\Delta 
G_{\cdot}\left(  u_{\sigma\left(  1\right)  }^{\prime}\right)  ;\cdots;\Delta 
G_{\cdot}\left(  u_{\sigma\left(  k\right)  }^{\prime}\right)  ;\Delta 
G_{\cdot}\left(  u_{\sigma\left(  k\right)  }^{\prime}\right)  \right]  , 
\label{crazycat}% 
\end{align} 
modulo the fact that one may remove the integrals with respect to those 
$u_{j}^{\prime}$'s not represented among $\{u_{\sigma\left(  1\right) 
}^{\prime},\cdots,u_{\sigma\left(  k\right)  }^{\prime}\}$. If we can show 
that for all $k\in\left\{  0,1,2,\cdots,\left[  m/2\right]  \right\}  $ and 
all $\sigma\in\Sigma_{m}^{k}$, the above expression $K=K_{m,k,\sigma}$ tends 
to $0$ as $\varepsilon$ tends to $0$, the theorem will be proved. 
 
A final note about notation. The bracket notation in the last two lines of the 
expression (\ref{crazycat}) above means that we have the product of two 
separate Riemann integrals over $s\in\lbrack0,T]$. Below we will denote these 
integrals as being with respect to $s\in\lbrack0,T]$ and $t\in\lbrack 
0,T]$.\bigskip 
 
\noindent\emph{Step 1: diagonal.} As in Step 1 of the proofs of Theorems 
\ref{HomogGauss} and \ref{nonhomoggauss}, we can Cauchy-Schwarz to deal with 
the portion of $K_{m,k,\sigma}$ in (\ref{crazycat}) where $\left\vert 
s-t\right\vert \leq2\varepsilon$. The details are omitted.\bigskip 
 
\noindent\emph{Step 2: term for }$k=0$. When $k=0$, there is only one 
permutation $\sigma=Id$, and we have, using hypothesis (A)% 
\begin{align*} 
K_{m,0,Id}  &  =\frac{1}{\varepsilon^{2}}\int_{u_{m}=0}^{T}\int_{u_{m-1}% 
=0}^{u_{m}}\cdots\int_{u_{1}=0}^{u_{2}}\mathbf{E}\left[  d\left[  M\right] 
^{\otimes m}\left(  u_{1},\cdots,u_{m}\right)  \right]  \cdot\left[  \Delta 
G_{\cdot}\left(  u_{1}\right)  ;\cdots;\Delta G_{\cdot}\left(  u_{m}\right) 
\right]  ^{2}\\ 
&  \leq\frac{1}{\varepsilon^{2}}\int_{u_{m-k}=0}^{T}\int_{u_{m-k-1}% 
=0}^{u_{m-k}}\cdots\int_{u_{1}=0}^{u_{2}}\Gamma^{2}\left(  u_{1}\right) 
\Gamma^{2}\left(  u_{2}\right)  \cdots\Gamma^{2}\left(  u_{m}\right)  \left[ 
\Delta G_{\cdot}\left(  u_{1}\right)  ;\cdots;\Delta G_{\cdot}\left( 
u_{m}\right)  \right]  ^{2}du_{1}du_{2}\cdots du_{m}\\ 
&  =\frac{1}{\varepsilon^{2}}\int_{u_{m-k}=0}^{T}\int_{u_{m-k-1}=0}^{u_{m-k}% 
}\cdots\int_{u_{1}=0}^{u_{2}}\left[  \Delta\tilde{G}_{\cdot}\left( 
u_{1}\right)  ;\cdots;\Delta\tilde{G}_{\cdot}\left(  u_{m}\right)  \right] 
^{2}du_{1}du_{2}\cdots du_{m}. 
\end{align*} 
This is precisely the expression one gets for the term corresponding to $k=0$ 
when $M=W$, i.e. when $X$ is the Gaussian process $Z$ with kernel $\tilde{G}$. 
Hence our hypotheses from the previous two theorems guarantee that this 
expression tends to $0$.\bigskip 
 
\noindent\emph{Step 3: term for }$k=1$. Again, in this case, $\sigma=Id$, and 
we thus have, using hypothesis (A), 
\begin{align*} 
&  K_{m,1,Id}=\frac{1}{\varepsilon^{2}}\int_{u_{m-1}=0}^{T}\int_{u_{m-2}% 
=0}^{u_{m-1}}\cdots\int_{u_{1}=0}^{u_{2}}\int_{u_{1}^{\prime}=0}^{u_{2}% 
}\mathbf{E}\left[  d\left[  M\right]  \left(  u_{1}\right)  d\left[  M\right] 
\left(  u_{1}^{\prime}\right)  d\left[  M\right]  ^{\otimes\left(  m-2\right) 
}\left(  u_{2},\cdots,u_{m-1}\right)  \right] \\ 
&  \cdot\left[  \Delta G_{\cdot}\left(  u_{2}\right)  ;\cdots;\Delta G_{\cdot 
}\left(  u_{m-1}\right)  ;\Delta G_{\cdot}\left(  u_{1}\right)  ;\Delta 
G_{\cdot}\left(  u_{1}\right)  \right]  \cdot\left[  \Delta G_{\cdot}\left( 
u_{2}\right)  ;\cdots;\Delta G_{\cdot}\left(  u_{m-1}\right)  ;\Delta 
G_{\cdot}\left(  u_{1}^{\prime}\right)  ;\Delta G_{\cdot}\left(  u_{1}% 
^{\prime}\right)  \right] \\ 
&  \leq\frac{1}{\varepsilon^{2}}\int_{u_{m-1}=0}^{T}\int_{u_{m-2}=0}^{u_{m-1}% 
}\cdots\int_{u_{1}=0}^{u_{2}}\int_{u_{1}^{\prime}=0}^{u_{2}}du_{1}% 
du_{1}^{\prime}du_{2}\cdots du_{m-1}\Gamma^{2}\left(  u_{1}\right)  \Gamma 
^{2}\left(  u_{1}^{\prime}\right)  \Gamma^{2}\left(  u_{2}\right) 
\cdots\Gamma^{2}\left(  u_{m}\right) \\ 
&  \cdot\left[  \left\vert \Delta G\right\vert _{\cdot}\left(  u_{2}\right) 
;\cdots;\left\vert \Delta G\right\vert _{\cdot}\left(  u_{m-1}\right) 
;\left\vert \Delta G\right\vert _{\cdot}\left(  u_{1}\right)  ;\left\vert 
\Delta G\right\vert _{\cdot}\left(  u_{1}\right)  \right]  \cdot\left[ 
\left\vert \Delta G\right\vert _{\cdot}\left(  u_{2}\right)  ;\cdots 
;\left\vert \Delta G\right\vert _{\cdot}\left(  u_{m-1}\right)  ;\left\vert 
\Delta G\right\vert _{\cdot}\left(  u_{1}^{\prime}\right)  ;\left\vert \Delta 
G\right\vert _{\cdot}\left(  u_{1}^{\prime}\right)  \right] \\ 
&  =\frac{1}{\varepsilon^{2}}\int_{u_{m-1}=0}^{T}\int_{u_{m-2}=0}^{u_{m-1}% 
}\cdots\int_{u_{1}=0}^{u_{2}}\int_{u_{1}^{\prime}=0}^{u_{2}}du_{1}% 
du_{1}^{\prime}du_{2}\cdots du_{m-1} \\
&  \left[  \left\vert \Delta\tilde 
{G}\right\vert _{\cdot}\left(  u_{2}\right)  ;\cdots;\left\vert \Delta 
\tilde{G}\right\vert _{\cdot}\left(  u_{m-1}\right)  ;  
   \left\vert \Delta 
\tilde{G}\right\vert _{\cdot}\left(  u_{1}\right)  ;\left\vert \Delta\tilde 
{G}\right\vert _{\cdot}\left(  u_{1}\right)  \right]  \\
&  \cdot\left[  \left\vert \Delta\tilde{G}\right\vert _{\cdot}\left( 
u_{2}\right)  ;\cdots;\left\vert \Delta\tilde{G}\right\vert _{\cdot}\left( 
u_{m-1}\right)  ;\left\vert \Delta\tilde{G}\right\vert _{\cdot}\left( 
u_{1}^{\prime}\right)  ;\left\vert \Delta\tilde{G}\right\vert _{\cdot}\left( 
u_{1}^{\prime}\right)  \right] 
\end{align*} 
Note now that the product of two bracket operators $\left[  \cdots\right] 
\left[  \cdots\right]  $ means we integrate over $0\leq s\leq t-2\varepsilon$ 
and $2\varepsilon\leq t\leq T$, and get an additional factor of $2$, since the 
diagonal term was dealt with in Step 1. 
 
In order to exploit the additional hypothesis (\ref{additional}) in our 
theorem, our first move is to use Fubini by bringing the integrals over 
$u_{1}$ all the way inside. We get% 
\begin{align*} 
K_{m,1,Id}  &  \leq\frac{2}{\varepsilon^{2}}\int_{u_{m-1}=0}^{T}\int 
_{u_{m-2}=0}^{u_{m-1}}\cdots\int_{u_{2}=0}^{u_{3}}du_{2}\cdots du_{m-1}\\ 
&  \int_{t=2\varepsilon}^{T}\int_{s=0}^{t-2\varepsilon}ds\ dt\left\vert 
\Delta\tilde{G}_{s}\left(  u_{2}\right)  \right\vert \cdots\left\vert 
\Delta\tilde{G}_{s}\left(  u_{m-1}\right)  \right\vert \left\vert \Delta 
\tilde{G}_{t}\left(  u_{2}\right)  \right\vert \cdots\left\vert \Delta 
\tilde{G}_{t}\left(  u_{m-1}\right)  \right\vert \\ 
&  \int_{u_{1}=0}^{u_{2}}\int_{u_{1}^{\prime}=0}^{u_{2}}du_{1}du_{1}^{\prime 
}\left(  \Delta\tilde{G}_{s}\left(  u_{1}\right)  \right)  ^{2}\left( 
\Delta\tilde{G}_{t}\left(  u_{1}^{\prime}\right)  \right)  ^{2}. 
\end{align*} 
The term in the last line above is trivially bounded above by% 
\[ 
\int_{u_{1}=0}^{T}\int_{u_{1}^{\prime}=0}^{T}du_{1}du_{1}^{\prime}\left( 
\Delta\tilde{G}_{s}\left(  u_{1}\right)  \right)  ^{2}\left(  \Delta\tilde 
{G}_{t}\left(  u_{1}^{\prime}\right)  \right)  ^{2}% 
\] 
precisely equal to $Var\left[  Z\left(  s+\varepsilon\right)  -Z\left( 
s\right)  \right]  \ Var\left[  Z\left(  t+\varepsilon\right)  -Z\left( 
t\right)  \right]  $, which by hypothesis is bounded above by $\delta 
^{4}\left(  \varepsilon\right)  $. Consequently, we get% 
\begin{align*} 
K_{m,1,Id}  &  \leq2\frac{\delta^{4}\left(  \varepsilon\right)  }% 
{\varepsilon^{2}}\int_{u_{m-1}=0}^{T}\int_{u_{m-2}=0}^{u_{m-1}}\cdots 
\int_{u_{2}=0}^{u_{3}}du_{2}\cdots du_{m-1}\\ 
&  \int_{t=2\varepsilon}^{T}\int_{s=0}^{t-2\varepsilon}ds\ dt\left\vert 
\Delta\tilde{G}_{s}\left(  u_{2}\right)  \right\vert \cdots\left\vert 
\Delta\tilde{G}_{s}\left(  u_{m-1}\right)  \right\vert \left\vert \Delta 
\tilde{G}_{t}\left(  u_{2}\right)  \right\vert \cdots\left\vert \Delta 
\tilde{G}_{t}\left(  u_{m-1}\right)  \right\vert . 
\end{align*} 
We get an upper bound by integrating all the $u_{j}$'s over their entire range 
$[0,T]$. I.e. we have,% 
\begin{align*} 
&  K_{m,1,Id}\leq\frac{\delta^{4}\left(  \varepsilon\right)  }{\varepsilon 
^{2}}\int_{t=2\varepsilon}^{T}dt\int_{s=0}^{t-2\varepsilon}ds\\ 
&  \int_{0}^{T}\int_{0}^{T}\cdots\int_{0}^{T}du_{3}\cdots du_{m-1}\left\vert 
\Delta\tilde{G}_{s}\left(  u_{3}\right)  \right\vert \cdots\left\vert 
\Delta\tilde{G}_{s}\left(  u_{m-1}\right)  \right\vert \left\vert \Delta 
\tilde{G}_{t}\left(  u_{3}\right)  \right\vert \cdots\left\vert \Delta 
\tilde{G}_{t}\left(  u_{m-1}\right)  \right\vert \\ 
&  \cdot\int_{u_{2}=0}^{T}\left\vert \Delta\tilde{G}_{t}\left(  u_{2}\right) 
\right\vert \left\vert \Delta\tilde{G}_{s}\left(  u_{2}\right)  \right\vert 
du_{2}\\ 
&  =2\frac{\delta^{4}\left(  \varepsilon\right)  }{\varepsilon^{2}}% 
\int_{t=2\varepsilon}^{T}dt\int_{s=0}^{t-2\varepsilon}ds\left(  \int_{0}% 
^{T}du\left\vert \Delta\tilde{G}_{s}\left(  u\right)  \right\vert \left\vert 
\Delta\tilde{G}_{t}\left(  u\right)  \right\vert \right)  ^{m-3}\cdot 
\int_{u_{2}=0}^{u_{3}}\left\vert \Delta\tilde{G}_{t}\left(  u_{2}\right) 
\right\vert \left\vert \Delta\tilde{G}_{s}\left(  u_{2}\right)  \right\vert 
du_{2}.. 
\end{align*} 
Now we use a simple Cauchy-Schwarz inequality for the integral over $u$, but 
not for $u_{2}$. Recognizing that $\int_{0}^{T}\left\vert \Delta\tilde{G}% 
_{s}\left(  u\right)  \right\vert ^{2}du$ is the variance $Var\left[  Z\left( 
s+\varepsilon\right)  -Z\left(  s\right)  \right]  \leq\delta^{2}\left( 
\varepsilon\right)  $, we have% 
\begin{align*} 
K_{m,1,Id}  &  \leq2\frac{\delta^{4}\left(  \varepsilon\right)  }% 
{\varepsilon^{2}}\int_{t=2\varepsilon}^{T}dt\int_{s=0}^{t-2\varepsilon 
}ds\left(  \int_{0}^{T}du\left\vert \Delta\tilde{G}_{s}\left(  u\right) 
\right\vert ^{2}\right)  ^{m-3}\cdot\int_{u_{2}=0}^{u_{3}}\left\vert 
\Delta\tilde{G}_{t}\left(  u_{2}\right)  \right\vert \left\vert \Delta 
\tilde{G}_{s}\left(  u_{2}\right)  \right\vert du_{2}.\\ 
&  \leq2\frac{\delta^{4+2m-6}\left(  \varepsilon\right)  }{\varepsilon^{2}% 
}\int_{t=2\varepsilon}^{T}dt\int_{s=0}^{t-2\varepsilon}ds\int_{u_{2}=0}% 
^{T}\left\vert \Delta\tilde{G}_{t}\left(  u_{2}\right)  \right\vert \left\vert 
\Delta\tilde{G}_{s}\left(  u_{2}\right)  \right\vert du_{2}. 
\end{align*} 
Condition (\ref{additional}) implies immediately $K_{m,1,Id}\leq\delta 
^{2m}\left(  2\varepsilon\right)  \varepsilon^{-1}$ which tends to $0$ with 
$\varepsilon$ by hypothesis.\bigskip 
 
\noindent\emph{Step 4: }$k\geq2$. This step proceeds using the same technique 
as Step 3. Fix $k\geq2$. Now for each given permutation $\sigma$, there are 
$k$ pairs of parameters of the type $\left(  u,u^{\prime}\right)  $. Each of 
these contributes precisely a term $\delta^{4}\left(  \varepsilon\right)  $, 
as in the previous step, i.e. $\delta^{4k}\left(  \varepsilon\right)  $ 
altogether. In other words, for every $\sigma\in\Sigma_{m}^{k}$, and deleting 
the diagonal term, we have 
\begin{align*} 
&  K_{m,k,\sigma}\\ 
&  \leq2\frac{\delta^{4k}\left(  \varepsilon\right)  }{\varepsilon^{2}}% 
\int_{t=2\varepsilon}^{T}dt\int_{s=0}^{t-2\varepsilon}ds\int_{0}^{T}\int 
_{0}^{u_{m-k}}\cdots\int_{0}^{u_{k+2}}du_{k+1}\cdots du_{m-k}\left[  \int 
_{0}^{T}ds\left\vert \Delta\tilde{G}_{s}\left(  u_{k+1}\right)  \right\vert 
\cdots\left\vert \Delta\tilde{G}_{s}\left(  u_{m-k}\right)  \right\vert 
\right]  ^{2}. 
\end{align*} 
Since $k\leq\left(  m-1\right)  /2$, there is at least one integral, the one 
in $u_{k+1}$, above. We treat all the remaining integrals, if any, over 
$u_{k+2},\cdots,u_{m-k}$ with Cauchy-Schwarz's inequality as in Step 3, 
yielding a contribution $\delta^{2\left(  m-2k-1\right)  }\left( 
\varepsilon\right)  $. The remaining integral over $u_{k+1}$ yields, by 
Condition (\ref{additional}), a contribution of $\delta^{2}\left( 
2\varepsilon\right)  \varepsilon$. Hence the contribution of $K_{m,k,\sigma}$ 
is again $\delta^{2m}\left(  2\varepsilon\right)  \varepsilon^{-1}$, which 
tends to $0$ with $\varepsilon$ by hypothesis, concluding the proof of the 
Theorem.$\vspace*{0.1in}$ 
\end{proof} 
 
\begin{proof} 
[Proof of Proposition \ref{Ex}]Below the value $1/\left(  2m\right)  -1/2$ is 
denoted by $\alpha$. We now show that we can apply Theorem \ref{nonhomoggauss} 
directly to the Gaussian process $Z$ given in (\ref{Zee}), which, by Theorem 
\ref{MartThm}, is sufficient, together with Condition (\ref{additional}), to 
obtain our desired conclusion. Note the assumption about $\tilde{G}$ implies 
that $s\mapsto\tilde{G}\left(  t,s\right)  $ is square-integrable, and 
therefore $Z$ is well-defined. We will prove Condition (\ref{nhdeltacond}) 
holds in Step 1; Step 2 will show Condition (\ref{nhmubirdiecond}) holds; 
Condition (\ref{additional}) will be established in Step 3.\vspace{0.1in} 
 
\noindent\emph{Step 1. Variance calculation. }We need only to show 
$\tilde{\delta}^{2}\left(  s,s+\varepsilon\right)  =o\left(  \varepsilon 
^{1/m}\right)  $ uniformly in $s$. We have, for given $s$ and $t=s+\varepsilon 
$% 
\begin{align} 
\tilde{\delta}^{2}\left(  s,s+\varepsilon\right)   &  =\int_{0}^{s}\left\vert 
\left(  s+\varepsilon-r\right)  ^{\alpha}f\left(  s+\varepsilon,r\right) 
-\left(  s-r\right)  ^{\alpha}f\left(  s,r\right)  \right\vert ^{2}% 
dr\nonumber\\ 
&  +\int_{s}^{s+\varepsilon}\left\vert s+\varepsilon-r\right\vert ^{2\alpha 
}f^{2}\left(  s+\varepsilon,r\right)  dr=:A+B \label{delta2withf}% 
\end{align} 
Since $f^{2}\left(  s+\varepsilon,r\right)  \leq f\left(  s+\varepsilon 
-r\right)  $ and the univariate $f$ increases, in $B$ we can bound this last 
quantity by $f\left(  \varepsilon\right)  $, and we get 
\[ 
B\leq f^{2}\left(  \varepsilon\right)  \int_{0}^{\varepsilon}r^{2\alpha 
}dr=3f^{2}\left(  \varepsilon\right)  \varepsilon^{2\alpha+1}=o\left( 
\varepsilon^{1/m}\right)  . 
\]

The term $A$ is slightly more delicate to estimate. Since $f$ is increasing 
and $g$ is decreasing in $t$,% 
\begin{align*} 
A  &  \leq\int_{0}^{s}f^{2}\left(  s+\varepsilon,r\right)  \left\vert \left( 
s+\varepsilon-r\right)  ^{\alpha}-\left(  s-r\right)  ^{\alpha}\right\vert 
^{2}dr=\int_{0}^{s}f^{2}\left(  \varepsilon+r\right)  \left\vert r^{\alpha 
}-\left(  r+\varepsilon\right)  ^{\alpha}\right\vert ^{2}dr\\ 
&  =\int_{0}^{\varepsilon}f^{2}\left(  \varepsilon+r\right)  \left\vert 
r^{\alpha}-\left(  r+\varepsilon\right)  ^{\alpha}\right\vert ^{2}% 
dr+\int_{\varepsilon}^{s}f^{2}\left(  \varepsilon+r\right)  \left\vert 
r^{\alpha}-\left(  r+\varepsilon\right)  ^{\alpha}\right\vert ^{2}dr\\ 
&  =:A_{1}+A_{2}. 
\end{align*} 
We have, again from the univariate $f$'s increasingness, and the limit 
$\lim_{r\rightarrow0}f\left(  r\right)  =0$,% 
\[ 
A_{1}\leq f^{2}\left(  2\varepsilon\right)  \int_{0}^{\varepsilon}\left\vert 
r^{\alpha}-\left(  r+\varepsilon\right)  ^{\alpha}\right\vert ^{2}dr=cst\cdot 
f^{2}\left(  2\varepsilon\right)  \varepsilon^{2\alpha+1}=o\left( 
\varepsilon^{1/m}\right)  . 
\] 
For the other part of $A$, we need to use $f$'s concavity at the point 
$2\varepsilon$ in the interval $[0,\varepsilon+r]$ (since $\varepsilon 
+r>2\varepsilon$ in this case), which implies $f\left(  \varepsilon+r\right) 
<f\left(  2\varepsilon\right)  \left(  \varepsilon+r\right)  /\left( 
2\varepsilon\right)  $. Also using the mean-value theorem for the difference 
of negative cubes, we get% 
\begin{align*} 
A_{2}  &  \leq cst\cdot\varepsilon^{2}\int_{\varepsilon}^{s}f^{2}\left( 
\varepsilon+r\right)  r^{2\alpha-2}dr\leq cst\cdot\varepsilon f\left( 
2\varepsilon\right)  \int_{\varepsilon}^{s}\left(  \varepsilon+r\right) 
r^{2\alpha-2}dr\\ 
&  \leq cst\cdot\varepsilon f\left(  2\varepsilon\right)  \int_{\varepsilon 
}^{s}r^{2\alpha-1}=cst\cdot\varepsilon^{2\alpha+1}f\left(  2\varepsilon 
\right)  =o\left(  \varepsilon^{1/3}\right)  . 
\end{align*} 
This finishes the proof of Condition (\ref{nhdeltacond}).\vspace{0.1in} 
 
\noindent\emph{Step 2. Covariance calculation}. We first calculate the second 
mixed derivative $\partial^{2}\tilde{\delta}^{2}/\partial s\partial t$, where 
$\tilde{\delta}$ is the canonical metric of $Z$, because we must show 
$\left\vert \mu\right\vert \left(  OD\right)  \leq\varepsilon^{2\alpha}$, 
which is condition (\ref{nhmubirdiecond}), and $\mu\left(  dsdt\right) 
=ds\ dt\ \partial^{2}\tilde{\delta}^{2}/\partial s\partial t$. We have, for 
$0\leq s\leq t-\varepsilon$,% 
\[ 
\tilde{\delta}^{2}\left(  s,t\right)  =\int_{0}^{s}\left(  g\left( 
t,s-r\right)  -g\left(  s,s-r\right)  \right)  ^{2}dr+\int_{s}^{t}g^{2}\left( 
t,r\right)  dr=:A+B. 
\] 
We calculate% 
\begin{align*} 
\frac{\partial^{2}A}{\partial s\partial t}\left(  t,s\right)   & 
=2\frac{\partial g}{\partial t}\left(  t,0\right)  \left(  g\left( 
t,0\right)  -g\left(  s,0\right)  \right) \\ 
&  +\int_{0}^{s}2\frac{\partial g}{\partial t}\left(  t,s-r\right)  \left( 
\frac{\partial g}{\partial s}\left(  t,s-r\right)  -\frac{\partial g}{\partial 
t}\left(  s,s-r\right)  -\frac{\partial g}{\partial s}\left(  s,s-r\right) 
\right) \\ 
&  +\int_{0}^{s}2\left(  g\left(  t,s-r\right)  -g\left(  s,s-r\right) 
\right)  \frac{\partial^{2}g}{\partial s\partial t}\left(  t,s-r\right)  dr.\\ 
&  =A_{1}+A_{2}+A_{3}, 
\end{align*} 
and 
\[ 
\frac{\partial^{2}B}{\partial s\partial t}\left(  t,s\right)  =-2g\left( 
t,s\right)  \frac{\partial g}{\partial t}\left(  t,s\right)  . 
\]

Next, we immediately get, for the portion of $\left\vert \mu\right\vert 
\left(  OD\right)  $ corresponding to $B$, using the hypotheses of our 
proposition,% 
\begin{align*} 
\int_{\varepsilon}^{T}dt\int_{0}^{t-\varepsilon}ds\left\vert \frac 
{\partial^{2}B}{\partial s\partial t}\left(  t,s\right)  \right\vert  & 
\leq2c\int_{\varepsilon}^{T}dt\int_{0}^{t-\varepsilon}dsf\left(  \left\vert 
t-s\right\vert \right)  \left\vert t-s\right\vert ^{\alpha}\left\vert 
t-s\right\vert ^{\alpha-1}\\ 
&  \leq2c\left\Vert f\right\Vert _{\infty}\int_{\varepsilon}^{T}% 
dt\ \varepsilon^{2\alpha}=cst\cdot\varepsilon^{2\alpha}, 
\end{align*} 
which is of the correct order for Condition (\ref{nhmubirdiecond}). For the 
term corresponding to $A_{1}$, using our hypotheses, we have% 
\[ 
\int_{\varepsilon}^{T}dt\int_{0}^{t-\varepsilon}ds\left\vert A_{1}\right\vert 
\leq2\int_{\varepsilon}^{T}dt\int_{0}^{t-\varepsilon}ds\ t^{\alpha}\left\vert 
\frac{\partial g}{\partial t}\left(  \xi_{t,s},0\right)  \right\vert 
\left\vert t-s\right\vert 
\] 
where $\xi_{t,s}$ is in the interval $\left(  s,t\right)  $. Our hypothesis 
thus implies $\left\vert \frac{\partial g}{\partial t}\left(  \xi 
_{t,s},0\right)  \right\vert \leq s^{\alpha}$, and hence% 
\[ 
\int_{\varepsilon}^{T}dt\int_{0}^{t-\varepsilon}ds\left\vert A_{1}\right\vert 
\leq2T\int_{\varepsilon}^{T}dt\int_{0}^{t-\varepsilon}ds\ s^{\alpha 
-1}t^{\alpha-1}=2T\alpha^{-1}\int_{\varepsilon}^{T}dt\ t^{\alpha-1}\left( 
t-\varepsilon\right)  ^{\alpha}\leq\alpha^{-2}T^{1+2\alpha}. 
\] 
This is much smaller than the right-hand side $\varepsilon^{2\alpha}$ of 
Condition (\ref{nhmubirdiecond}), since $2\alpha=1/m-1<0$. The terms $A_{2}$ 
and $A_{3}$ are treated similarly, thanks to our hypotheses.\vspace{0.1in} 
 
\noindent\emph{Step 3: proving Condition (\ref{additional}).} We modify the 
proof of Theorem \ref{MartThm}, in particular Steps 3 and 4, so that we only 
need to prove% 
\begin{equation} 
\int_{t=2\varepsilon}^{T}dt\int_{s=0}^{t-2\varepsilon}ds\int_{u=0}% 
^{T}\left\vert \Delta\tilde{G}_{t}\left(  u\right)  \right\vert \left\vert 
\Delta\tilde{G}_{s}\left(  u\right)  \right\vert du\leq c\varepsilon 
^{2+2\alpha}=c\varepsilon^{1/m+1}, \label{additional2}% 
\end{equation} 
instead of Condition (\ref{additional}). Indeed, for instance in Step 3, this 
new condition yields a final contribution of order $\delta^{2m-2}\left( 
\varepsilon\right)  \varepsilon^{-2}\varepsilon^{1/m+1}$. With the assumption 
on $\delta$ that we have, $\delta\left(  \varepsilon\right)  =o\left( 
\varepsilon^{1/\left(  2m\right)  }\right)  $, and hence the final 
contribution is of order $o\left(  \varepsilon^{\left(  2m-2\right)  /\left( 
2m\right)  -1+1/m}\right)  =o\left(  1\right)  $. This proves that the 
conclusion of Theorem \ref{MartThm} holds if we assume (\ref{additional2}) 
instead of Condition (\ref{additional}). 
 
We now prove (\ref{additional2}). We can write% 
\begin{align*} 
&  \int_{t=2\varepsilon}^{T}dt\int_{s=0}^{t-2\varepsilon}ds\int_{u=0}% 
^{T}\left\vert \Delta\tilde{G}_{t}\left(  u\right)  \right\vert \left\vert 
\Delta\tilde{G}_{s}\left(  u\right)  \right\vert du\\ 
&  =\int_{t=2\varepsilon}^{T}dt\int_{s=0}^{t-2\varepsilon}ds\int_{0}% 
^{s}\left\vert g\left(  t+\varepsilon,u\right)  -g\left(  t,u\right) 
\right\vert \left\vert g\left(  s+\varepsilon,u\right)  -g\left(  s,u\right) 
\right\vert du\\ 
&  +\int_{t=2\varepsilon}^{T}dt\int_{s=0}^{t-2\varepsilon}ds\int 
_{s}^{s+\varepsilon}\left\vert g\left(  t+\varepsilon,u\right)  -g\left( 
t,u\right)  \right\vert \left\vert g\left(  s+\varepsilon,u\right) 
\right\vert du=:A+B. 
\end{align*} 
For $A$, we use the hypotheses of this proposition: for the last factor in 
$A$, we exploit the fact that $g$ is decreasing in $t$ while $f$ is increasing 
in $t$; for the other factor in\ $A$, use the bound on $\partial g/\partial 
t$; thus we have% 
\[ 
A\leq\int_{t=2\varepsilon}^{T}dt\int_{s=0}^{t-2\varepsilon}\varepsilon 
\left\vert t-s\right\vert ^{\alpha-1}ds\int_{0}^{s}f\left(  s+\varepsilon 
,u\right)  \left(  \left(  s-u\right)  ^{\alpha}-\left(  s+\varepsilon 
-u\right)  ^{\alpha}\right)  du. 
\] 
We separate the integral in $u$ into two pieces, for $u\in\lbrack 
0,s-\varepsilon]$ and $u\in\lbrack s-\varepsilon,s]$. For the first integral 
in $u$, since $f$ is bounded, we have% 
\[ 
\int_{0}^{s-\varepsilon}f\left(  s+\varepsilon,u\right)  \left(  \left( 
s-u\right)  ^{\alpha}-\left(  s+\varepsilon-u\right)  ^{\alpha}\right) 
du\leq\left\Vert f\right\Vert _{\infty}\varepsilon\int_{0}^{s-\varepsilon 
}\left(  s-u\right)  ^{\alpha-1}du\leq\left\Vert f\right\Vert _{\infty 
}c_{\alpha}\varepsilon^{1+\alpha}. 
\] 
For the second integral in $u$, we use the fact that $s-u+\varepsilon 
>\varepsilon$ and $s-u<\varepsilon$ implies $s-u+\varepsilon>2\left( 
s-u\right)  $, so that the negative part of the integral can be ignored, and 
thus% 
\[ 
\int_{s-\varepsilon}^{s}f\left(  s+\varepsilon,u\right)  \left(  \left( 
s-u\right)  ^{\alpha}-\left(  s+\varepsilon-u\right)  ^{\alpha}\right) 
du\leq\left\Vert f\right\Vert _{\infty}\int_{s-\varepsilon}^{s}\left( 
s-u\right)  ^{\alpha}du=\left\Vert f\right\Vert _{\infty}c_{\alpha}% 
\varepsilon^{1+\alpha}, 
\] 
which is the same upper bound as for the other part of the integral in $u$. 
Thus% 
\[ 
A\leq cst\cdot\varepsilon^{2+\alpha}\int_{t=2\varepsilon}^{T}dt\int 
_{s=0}^{t-2\varepsilon}\left\vert t-s\right\vert ^{\alpha-1}ds\leq 
cst\cdot\varepsilon^{2+\alpha}\int_{t=2\varepsilon}^{T}dt\ \varepsilon 
^{\alpha}\leq cst\cdot\varepsilon^{2+2\alpha}=cst\cdot\varepsilon^{1/m+1}, 
\] 
which is the conclusion we needed at least for $A$. 
 
Lastly, we estimate $B$. We use the fact that $f$ is bounded, and thus 
$\left\vert g\left(  s+\varepsilon,u\right)  \right\vert \leq\left\Vert 
f\right\Vert _{\infty}\left\vert s+\varepsilon-u\right\vert ^{\alpha}$, as 
well as the estimate on the derivative of $g$ as we did in the calculation of 
$A$, yielding 
\begin{align*} 
B  &  \leq\left\Vert f\right\Vert _{\infty}\varepsilon\int_{t=2\varepsilon 
}^{T}dt\int_{s=0}^{t-2\varepsilon}ds\left\vert t-s-\varepsilon\right\vert 
^{\alpha-1}\int_{s}^{s+\varepsilon}\left\vert s+\varepsilon-u\right\vert 
^{\alpha}du\\ 
&  =cst\cdot\varepsilon^{\alpha+2}\int_{t=2\varepsilon}^{T}dt\int 
_{s=0}^{t-2\varepsilon}ds\left\vert t-s-\varepsilon\right\vert ^{\alpha-1}\\ 
&  \leq2^{1+\left\vert \alpha\right\vert }cst\cdot\varepsilon^{\alpha+2}% 
\int_{t=2\varepsilon}^{T}dt\int_{s=0}^{t-2\varepsilon}ds\left\vert 
t-s\right\vert ^{\alpha-1}\leq cst\cdot\varepsilon^{2\alpha+2}=cst\cdot 
\varepsilon^{1/m+1}. 
\end{align*} 
This is the conclusion we needed for $B,$which finishes the proof of the 
proposition.$\vspace*{0.1in}$ 
\end{proof} 
 
\begin{proof} 
[Proof of Lemma \ref{Lemma53}] 
 
\emph{Step 1: Setup. }We only need to show that for all $i,j\in\left\{ 
1,2\right\}  $,% 
\begin{equation} 
\iint_{D_{\varepsilon}}\left\vert r_{ij}\right\vert ^{k}dudv\leq 
cst\cdot\varepsilon\delta^{k}\left(  \varepsilon\right)  . \label{intrijk}% 
\end{equation} 
Recall the function $K$ defined in \cite{GNRV} 
\begin{align*} 
K\left(  u,v\right)   &  :=\mathbf{E}\left[  \left(  X_{u+\varepsilon}% 
+X_{u}\right)  \left(  X_{v+\varepsilon}+X_{v}\right)  \right] \\ 
&  =Q\left(  u+\varepsilon,v+\varepsilon\right)  +Q\left(  u,v+\varepsilon 
\right)  +Q\left(  u+\varepsilon,v\right)  +Q\left(  u,v\right)  . 
\end{align*} 
This is not to be confused with the usage of the letter $K$ in previous 
sections, to which there will be made no reference in this proof; the same 
remark hold for the notation $\Delta$ borrowed again from \cite{GNRV}, and 
used below. 
 
To follow the proof in \cite{GNRV}, we need to prove the following items for 
some constants $c_{1}$ and $c_{2}$: 
 
\begin{enumerate} 
\item $c_{1}\delta^{2}\left(  u\right)  \leq K\left(  u,u\right)  \leq 
c_{2}\delta^{2}\left(  u\right)  ;$ 
 
\item $K\left(  u,v\right)  \leq c_{2}\delta\left(  u\right)  \delta\left( 
v\right)  ;$ 
 
\item $\Delta\left(  u,v\right)  :=K\left(  u,u\right)  K\left(  v,v\right) 
-K\left(  u,v\right)  ^{2}\geq c_{1}\delta^{2}\left(  u\right)  \delta 
^{2}\left(  v-u\right)  .$ 
\end{enumerate} 
 
By the Theorem's upper bound assumption on the bivariate $\delta^{2}$ 
(borrowed from Theorem \ref{nonhomoggauss}), its assumptions on the 
monotonicity of $Q$ and the univariate $\delta$, and finally using the 
coercivity assumption (i), we have% 
\begin{align*} 
K\left(  u,u\right)   &  =Q_{u}+Q_{u+\varepsilon}+2Q\left(  u,u+\varepsilon 
\right)  =2\left(  Q_{u}+Q_{u+\varepsilon}\right)  -\delta^{2}\left( 
u,u+\varepsilon\right) \\ 
&  \geq2\left(  Q_{u}+Q_{u+\varepsilon}\right)  -\delta^{2}\left( 
\varepsilon\right)  \geq\left(  4-c^{-1}\right)  Q_{u}. 
\end{align*} 
This proves the lower bound in Item 1 above. The upper bound in Item 1 is a 
special case of Item 2, which we now prove. Again, the assumption borrowed 
from Theorem \ref{nonhomoggauss}, which says that $\delta^{2}\left( 
s,t\right)  \leq\delta^{2}\left(  \left\vert t-s\right\vert \right)  $, now 
implies, for $s=0$, that% 
\begin{equation} 
\delta^{2}\left(  0,u\right)  =Q_{u}\leq\delta^{2}\left(  u\right)  . 
\label{Quudelta}% 
\end{equation} 
We write, via Cauchy-Schwarz's inequality and the fact that $\delta^{2}$ is 
increasing, and thanks to (\ref{Quudelta}),% 
\[ 
K\left(  u,v\right)  \leq4\delta\left(  u+\varepsilon\right)  \delta\left( 
v+\varepsilon\right)  . 
\] 
However, since $\delta^{2}$ is concave with $\delta\left(  0\right)  =0$, we 
have $\delta^{2}\left(  2u\right)  /2u\leq\delta^{2}\left(  u\right)  /u$. 
Also, since we are in the set $D_{\varepsilon}$, $u+\varepsilon\leq2u$ and 
$v+\varepsilon\leq2v$. Hence% 
\[ 
K\left(  u,v\right)  \leq4\delta\left(  2u\right)  \delta\left(  2v\right) 
\leq8\delta\left(  u\right)  \delta\left(  v\right)  , 
\] 
which is Item 2. 
 
We now verify Item 3 for all $u,v\in D_{\varepsilon}$ , assuming in addition 
that $v$ is not too small, specifically $v>\varepsilon^{\rho/2}$. One can 
estimate the integral in Lemma \ref{Lemma53} restricted to those values where 
$v\leq\varepsilon^{\rho/2}$ using coarser tools than we use below; we omit the 
corresponding calculations. From the definition of $K$ above, using the fact 
that, by our concavity assumptions, $Q$ is, in both variables, a sum of 
Lipschitz functions, we have, for small $\varepsilon$,% 
\[ 
K\left(  u,v\right)  =4Q\left(  u,v\right)  +O\left(  \varepsilon\right)  . 
\] 
Therefore, 
\[ 
\Delta=16\left(  Q_{u}Q_{v}-Q^{2}\left(  u,v\right)  \right)  +O\left( 
\varepsilon\right)  . 
\] 
Assumption (ii) in the Theorem now implies% 
\[ 
\Delta\geq16c^{\prime}\delta^{2}\left(  u\right)  \delta^{2}\left( 
v-u\right)  +O\left(  \varepsilon\right)  . 
\] 
The concavity of $Q$ and Assumption (i) imply $\delta^{2}\left(  r\right) 
\geq Q_{r}\geq cst\cdot r$. Moreover, because of the restriction on $v$, 
either $v-u>cst\cdot\varepsilon^{\rho/2}$ or $u>cst\cdot\varepsilon^{\rho/2}$. 
Therefore $\delta^{2}\left(  u\right)  \delta^{2}\left(  v-u\right)  \geq 
cst\cdot\varepsilon^{1-\rho}\varepsilon^{\rho/2}\gg\varepsilon$. Therefore, 
for $\varepsilon$ small enough, $\Delta\geq8c^{\prime}\delta^{2}\left( 
u\right)  \delta^{2}\left(  v-u\right)  $, proving Item 3.\bigskip 
 
It will now be necessary to reestimate the components of the matrix 
$\Lambda_{21}$ where we recall% 
\begin{align*} 
\Lambda_{21}[11]  &  :=\mathbf{E}\left[  \left(  X_{u+\varepsilon}% 
+X_{u}\right)  \left(  X_{u+\varepsilon}-X_{u}\right)  \right]  ,\\ 
\Lambda_{21}[12]  &  :=\mathbf{E}\left[  \left(  X_{v+\varepsilon}% 
+X_{v}\right)  \left(  X_{u+\varepsilon}-X_{u}\right)  \right]  ,\\ 
\Lambda_{21}[21]  &  :=\mathbf{E}\left[  \left(  X_{u+\varepsilon}% 
+X_{u}\right)  \left(  X_{v+\varepsilon}-X_{v}\right)  \right]  ,\\ 
\Lambda_{21}[22]  &  :=\mathbf{E}\left[  \left(  X_{v+\varepsilon}% 
+X_{v}\right)  \left(  X_{v+\varepsilon}-X_{v}\right)  \right]  . 
\end{align*}

\emph{Step 2: the term }$r_{11}$. We have by the lower bound of item 1 above 
on $K\left(  u,u\right)  $,% 
\[ 
\left\vert r_{11}\right\vert =\left\vert \frac{1}{\sqrt{K\left(  u,u\right) 
}}\Lambda_{21}[11]\right\vert \leq\frac{cst}{\delta\left(  u\right) 
}\left\vert \Lambda_{21}[11]\right\vert . 
\] 
To bound $\left\vert \Lambda_{21}[11]\right\vert $ above, we write% 
\begin{align*} 
\left\vert \Lambda_{21}[11]\right\vert  &  =\left\vert \mathbf{E}\left[ 
\left(  X_{u+\varepsilon}+X_{u}\right)  \left(  X_{u+\varepsilon}% 
-X_{u}\right)  \right]  \right\vert \\ 
&  =Q_{u+\varepsilon}-Q_{u}\leq\varepsilon\delta^{2}\left(  u\right)  /u 
\end{align*} 
where we used the facts that $Q_{u}$ is increasing and concave, and that 
$Q_{u}\leq\delta^{2}\left(  u\right)  $. Thus we have% 
\[ 
\left\vert r_{11}\right\vert \leq\varepsilon~cst\frac{\delta\left(  u\right) 
}{u}. 
\] 
The result (\ref{intrijk}) for $i=j=1$ now follows by the next lemma. 
 
\begin{lemma} 
\label{deltauuk}For every $k\geq2$, there exists $c_{k}>0$ such that for every 
$\varepsilon\in(0,1)$, $\int_{\varepsilon}^{1}\left\vert \delta\left( 
u\right)  /u\right\vert ^{k}du\leq c_{k}\varepsilon\left\vert \delta\left( 
\varepsilon\right)  /\varepsilon\right\vert ^{k}$. 
\end{lemma} 
 
\textbf{Proof of lemma \ref{deltauuk}.} Our hypothesis (iii) can be rewritten 
as% 
\[ 
\frac{\delta\left(  au\right)  }{au}<\left(  \frac{1+\left(  a-1\right)  b}% 
{a}\right)  \frac{\delta\left(  u\right)  }{u}=:K_{a,b}\frac{\delta\left( 
u\right)  }{u}. 
\] 
The concavity of $\delta$ also implies that $\delta\left(  u\right)  /u$ is 
increasing. Thus we can write% 
\begin{align*} 
\int_{\varepsilon}^{1}\left\vert \frac{\delta\left(  u\right)  }{u}\right\vert 
^{k}du  &  \leq\sum_{j=0}^{\infty}\int_{\varepsilon a^{j}}^{\varepsilon 
a^{j+1}}\left\vert \frac{\delta\left(  u\right)  }{u}\right\vert ^{k}% 
du\leq\sum_{j=0}^{\infty}\left(  \varepsilon a^{j+1}-\varepsilon a^{j}\right) 
\left\vert K_{a,b}\right\vert ^{jk}\left\vert \frac{\delta\left( 
\varepsilon\right)  }{\varepsilon}\right\vert ^{k}\\ 
&  =\varepsilon\left(  a-1\right)  \left\vert \frac{\delta\left( 
\varepsilon\right)  }{\varepsilon}\right\vert ^{k}\sum_{j=0}^{\infty}\left( 
\left\vert K_{a,b}\right\vert ^{k}a\right)  ^{j}. 
\end{align*} 
The lemma will be proved if we can show that $f\left(  a\right)  :=\left\vert 
K_{a,b}\right\vert ^{k}a<1$ for some $a>1$. We have $f\left(  1\right)  =0$ 
and $f^{\prime}\left(  1\right)  =k\left(  1-b\right)  -1$. This last quantity 
is strictly positive for all $k\geq2$ as soon as $b<1/2$. This finishes the 
proof of the lemma \ref{deltauuk}.\hfill$\square$\vspace{0.15in} 
 
\emph{Step 3: the term }$r_{12}$. We have 
\[ 
r_{12}=\Lambda_{21}\left[  11\right]  \frac{-K\left(  u,v\right)  }% 
{\sqrt{K\left(  u,u\right)  \Delta\left(  u,v\right)  }}+\Lambda_{21}\left[ 
12\right]  \frac{\sqrt{K\left(  u,u\right)  }}{\sqrt{\Delta\left(  u,v\right) 
}}. 
\] 
We saw in the previous step that $\left\vert \Lambda_{21}\left[  11\right] 
\right\vert =\left\vert Q_{u+\varepsilon}-Q_{u}\right\vert \leq cst\cdot 
\varepsilon\delta^{2}\left(  u\right)  /u$. For $\Lambda_{21}\left[ 
12\right]  $, using the hypotheses on our increasing and concave functions, we 
calculate% 
\begin{align} 
\left\vert \Lambda_{21}\left[  12\right]  \right\vert  &  =\left\vert 2\left( 
Q_{u+\varepsilon}-Q_{u}\right)  +\delta^{2}\left(  u+\varepsilon 
,v+\varepsilon\right)  -\delta^{2}\left(  u,v+\varepsilon\right)  +\delta 
^{2}\left(  u+\varepsilon,v\right)  -\delta^{2}\left(  u,v\right)  \right\vert 
\nonumber\\ 
&  \leq2\left\vert \Lambda_{21}\left[  11\right]  \right\vert +\varepsilon 
\delta^{2}\left(  u+\varepsilon,v+\varepsilon\right)  /\left(  v-u\right) 
+\varepsilon\delta^{2}\left(  u+\varepsilon,v\right)  /\left(  v-u-\varepsilon 
\right) \nonumber\\ 
&  \leq2\left\vert \Lambda_{21}\left[  11\right]  \right\vert +\varepsilon 
\delta^{2}\left(  v-u\right)  /\left(  v-u\right)  +\varepsilon\delta 
^{2}\left(  v-u-\varepsilon\right)  /\left(  v-u-\varepsilon\right) 
\nonumber\\ 
&  \leq2cst\cdot\varepsilon\delta^{2}\left(  u\right)  /u+2\varepsilon 
\delta^{2}\left(  v-u-\varepsilon\right)  /\left(  v-u-\varepsilon\right)  . 
\label{Lambda12}% 
\end{align} 
The presence of the term $-\varepsilon$ in the last expression above is 
slightly aggravating, and one would like to dispose of it. However, since 
$\left(  u,v\right)  \in D_{\varepsilon}$, we have $v-u>\varepsilon^{\rho}$ 
for some $\rho\in\left(  0,1\right)  $. Therefore $v-u-\varepsilon 
>\varepsilon^{\rho}-\varepsilon>\varepsilon^{\rho/2}$ for $\varepsilon$ small 
enough. Hence by using $\rho/2$ instead of $\rho$ in the definition of 
$D_{\varepsilon}$ in the current calculation, we can ignore the term 
$-\varepsilon$ in the last displayed line above. Together with items 1, 2, and 
3 above which enable us to control the terms $K$ and $\Delta$ in $r_{12}$, we 
now have% 
\begin{align*} 
\left\vert r_{12}\right\vert  &  \leq cst\cdot\varepsilon\frac{\delta 
^{2}\left(  u\right)  }{u}\left(  \frac{\delta\left(  u\right)  \delta\left( 
v\right)  }{\delta\left(  u\right)  \delta\left(  u\right)  \delta\left( 
v-u\right)  }+\frac{\delta\left(  u\right)  }{\delta\left(  u\right) 
\delta\left(  v-u\right)  }\right) \\ 
&  +cst\cdot\varepsilon\frac{\delta^{2}\left(  v-u\right)  }{v-u}\frac 
{\delta\left(  u\right)  }{\delta\left(  u\right)  \delta\left(  v-u\right) 
}\\ 
&  =cst\cdot\varepsilon~\left(  \frac{\delta\left(  u\right)  \delta\left( 
v\right)  }{u\delta\left(  v-u\right)  }+\frac{\delta^{2}\left(  u\right) 
}{u\delta\left(  v-u\right)  }+\frac{\delta\left(  v-u\right)  }{v-u}\right) 
. 
\end{align*} 
We may thus write% 
\[ 
\iint_{D_{\varepsilon}}\left\vert r_{12}\right\vert ^{k}dudv\leq 
cst\cdot\varepsilon^{k}~\iint_{D_{\varepsilon}}\left(  \left\vert \frac 
{\delta\left(  u\right)  \delta\left(  v\right)  }{u\delta\left(  v-u\right) 
}\right\vert ^{k}+\left\vert \frac{\delta^{2}\left(  u\right)  }% 
{u\delta\left(  v-u\right)  }\right\vert ^{k}+\left\vert \frac{\delta\left( 
v-u\right)  }{v-u}\right\vert ^{k}\right)  dudv. 
\] 
The last term $\iint_{D_{\varepsilon}}\left\vert \frac{\delta\left( 
v-u\right)  }{v-u}\right\vert ^{k}dudv$ is identical, after a trivial change 
of variables, to the one dealt with in Step 2. Since $\delta$ is increasing, 
second the term $\iint_{D_{\varepsilon}}\left\vert \frac{\delta^{2}\left( 
u\right)  }{u\delta\left(  v-u\right)  }\right\vert ^{k}dudv$ is smaller than 
the first term $\iint_{D_{\varepsilon}}\left\vert \frac{\delta\left( 
u\right)  \delta\left(  v\right)  }{u\delta\left(  v-u\right)  }\right\vert 
^{k}dudv$. Thus we only need to deal with that first term; it is more delicate 
than what we estimated in Step 2. 
 
We separate the integral over $u$ at the intermediate point $v/2$. When 
$u\in\lbrack v/2,v-\varepsilon]$, we use the estimate% 
\[ 
\frac{\delta\left(  u\right)  }{u}\leq\frac{\delta\left(  v/2\right)  }% 
{v/2}\leq2\frac{\delta\left(  v\right)  }{v}. 
\] 
On the other hand when $u\in\lbrack\varepsilon,v/2]$ we simply bound 
$1/\delta\left(  v-u\right)  $ by $1/\delta\left(  v/2\right)  $. Thus% 
\begin{align*} 
&  \iint_{D_{\varepsilon}}\left\vert \frac{\delta\left(  u\right) 
\delta\left(  v\right)  }{u\delta\left(  v-u\right)  }\right\vert 
^{k}dudv=\int_{v=2\varepsilon}^{1}dv\int_{u=\varepsilon}^{v/2}\left\vert 
\frac{\delta\left(  u\right)  \delta\left(  v\right)  }{u\delta\left( 
v-u\right)  }\right\vert ^{k}du+\int_{v=\varepsilon}^{1}dv\int_{u=v/2}% 
^{v-\varepsilon}\left\vert \frac{\delta\left(  u\right)  \delta\left( 
v\right)  }{u\delta\left(  v-u\right)  }\right\vert ^{k}du\\ 
&  \leq\int_{v=2\varepsilon}^{1}dv\left\vert \frac{\delta\left(  v\right) 
}{\delta\left(  v/2\right)  }\right\vert ^{k}\int_{u=\varepsilon}% 
^{v/2}\left\vert \frac{\delta\left(  u\right)  }{u}\right\vert ^{k}% 
du+2\int_{v=\varepsilon}^{1}\left\vert \frac{\delta^{2}\left(  v\right)  }% 
{v}\right\vert ^{k}dv\int_{u=v/2}^{v-\varepsilon}\left\vert \frac{1}% 
{\delta\left(  v-u\right)  }\right\vert ^{k}du\\ 
&  \leq2^{k}\int_{u=\varepsilon}^{1}\left\vert \frac{\delta\left(  u\right) 
}{u}\right\vert ^{k}du+2\frac{1}{\delta^{k}\left(  \varepsilon\right)  }% 
\int_{v=\varepsilon}^{1}v^{k}\left\vert \frac{\delta\left(  v\right)  }% 
{v}\right\vert ^{2k}dv\\ 
&  \leq cst\cdot\varepsilon\left(  \frac{\delta\left(  \varepsilon\right) 
}{\varepsilon}\right)  ^{k}; 
\end{align*} 
here we used the concavity of $\delta$ to imply that $\delta\left(  v\right) 
/\delta\left(  v/2\right)  \leq2$, and to obtain the last line, we used Lemma 
\ref{deltauuk} for the first term in the previous line, and we used the fact 
that $\delta$ is increasing and that $v\leq1$, together again with Lemma 
\ref{deltauuk} for the second term in the previous line. This finishes the 
proof of (\ref{intrijk}) for $r_{12}.$\vspace{0.15in} 
 
\emph{Step 4: the term }$r_{21}$. We have% 
\[ 
r_{21}=\Lambda_{21}\left[  21\right]  \frac{1}{\sqrt{K\left(  u,u\right)  }}% 
\] 
and similarly to the previous step,% 
\begin{align*} 
\left\vert \Lambda_{21}\left[  21\right]  \right\vert  &  =\left\vert Q\left( 
u+\varepsilon,v+\varepsilon\right)  -Q\left(  u+\varepsilon,v\right) 
+Q\left(  u,v+\varepsilon\right)  -Q\left(  u,v\right)  \right\vert \\ 
&  =\left\vert 2\left(  Q_{v+\varepsilon}-Q_{v}\right)  +\delta^{2}\left( 
u+\varepsilon,v\right)  -\delta^{2}\left(  u+\varepsilon,v+\varepsilon\right) 
+\delta^{2}\left(  u,v\right)  -\delta^{2}\left(  u,v+\varepsilon\right) 
\right\vert \\ 
&  \leq2\left\vert \Lambda_{21}\left[  11\right]  \right\vert +\varepsilon 
\frac{\delta^{2}\left(  u+\varepsilon,v\right)  }{v-u-\varepsilon}% 
+\varepsilon\frac{\delta^{2}\left(  u,v\right)  }{v-u}\\ 
&  \leq2cst\cdot\varepsilon\delta^{2}\left(  u\right)  /u+4\varepsilon 
\delta^{2}\left(  v-u\right)  /\left(  v-u\right)  , 
\end{align*} 
which is the same expression as in (\ref{Lambda12}). Hence with the lower 
bound of Item 1 on $K\left(  u,u\right)  $ we have% 
\begin{align*} 
\iint_{D_{\varepsilon}}\left\vert r_{21}\right\vert ^{k}dudv  &  \leq 
cst\cdot\varepsilon^{k}~\iint_{D_{\varepsilon}}\left(  \left\vert \frac 
{\delta\left(  u\right)  }{u}\right\vert ^{k}+\left\vert \frac{\delta 
^{2}\left(  v-u\right)  }{\left(  v-u\right)  \delta\left(  u\right) 
}\right\vert ^{k}\right)  dudv\\ 
&  =cst\cdot\varepsilon^{k}~\iint_{D_{\varepsilon}}\left(  \left\vert 
\frac{\delta\left(  u\right)  }{u}\right\vert ^{k}+\left\vert \frac{\delta 
^{2}\left(  u\right)  }{u\delta\left(  v-u\right)  }\right\vert ^{k}\right) 
dudv. 
\end{align*} 
This is bounded above by the expression obtained as an upper bound in Step 3 
for $\iint_{D_{\varepsilon}}\left\vert r_{12}\right\vert ^{k}dudv$, which 
finishes the proof of (\ref{intrijk}) for $r_{21}.$\vspace{0.15in} 
 
\emph{Step 5: the term }$r_{22}$. Here we have% 
\[ 
r_{22}=\Lambda_{21}\left[  21\right]  \frac{-K\left(  u,v\right)  }% 
{\sqrt{K\left(  u,u\right)  \Delta\left(  u,v\right)  }}+\Lambda_{21}\left[ 
22\right]  \frac{\sqrt{K\left(  u,u\right)  }}{\sqrt{\Delta\left(  u,v\right) 
}}. 
\] 
We have already seen in the previous step that 
\[ 
\left\vert \Lambda_{21}\left[  21\right]  \right\vert \leq cst\cdot 
\varepsilon\left(  \frac{\delta^{2}\left(  u\right)  }{u}+\frac{\delta 
^{2}\left(  v-u\right)  }{v-u}\right)  . 
\] 
Moreover, we have, as in Step 2,% 
\[ 
\left\vert \Lambda_{21}\left[  22\right]  \right\vert =\left\vert 
Q_{v+\varepsilon}-Q_{v}\right\vert \leq cst\cdot\varepsilon\frac{\delta 
^{2}\left(  v\right)  }{v}. 
\] 
Thus using the bounds in items 1, 2, and 3,% 
\begin{align*} 
\left\vert r_{22}\right\vert  &  \leq cst\cdot\varepsilon\left[  \left( 
\frac{\delta^{2}\left(  u\right)  }{u}+\frac{\delta^{2}\left(  v-u\right) 
}{v-u}\right)  \frac{\delta\left(  u\right)  \delta\left(  v\right)  }% 
{\delta^{2}\left(  u\right)  \delta\left(  v-u\right)  }+\frac{\delta 
^{2}\left(  v\right)  }{v}\frac{\delta\left(  u\right)  }{\delta\left( 
u\right)  \delta\left(  v-u\right)  }\right] \\ 
&  =cst\cdot\varepsilon\left[  \frac{\delta\left(  u\right)  \delta\left( 
v\right)  }{u\delta\left(  v-u\right)  }+\frac{\delta\left(  v\right) 
\delta\left(  v-u\right)  }{\delta\left(  u\right)  \left(  v-u\right) 
}+\frac{\delta^{2}\left(  v\right)  }{v\delta\left(  v-u\right)  }\right]  . 
\end{align*} 
Of the last three terms, the first term was already treated in Step 3, the 
second is, up to a change of variable, identical to the first, and the third 
is smaller than $\frac{\delta^{2}\left(  u\right)  }{u\delta\left( 
v-u\right)  }$ which was also treated in Step 3. Thus (\ref{intrijk}) is 
proved for $r_{22}$, which finishes the entire proof of Lemma \ref{Lemma53}. 
%and it was partially 
%carried out during his visit to the Bernoulli Center (EPFL Lausanne). 
 
\end{proof}


\begin{thebibliography}{99}                                                                                               % 
 
 
\bibitem {A}Adler, R.: \emph{An introduction to continuity, extrema, and 
related topics for general Gaussian processes.} Inst. Math. Stat., Hayward, 
CA, (1990). 
 
\bibitem {AMN1}Al\`{o}s, E., Mazet, O., Nualart, D.: \emph{Stochastic calculus 
with respect to Gaussian processes.} Ann. Probab. \textbf{29}, 766--801 (1999) 
 
\bibitem {alos}Al\`{o}s, E. and Nualart, D.: \emph{Stochastic integration with 
respect to fractional Brownian motion.} Stoch. Stoch. Rep., \textbf{75}, 3, 
129--152 (2003). 
 
\bibitem {AW}Aza\"{\i}s, J.-M.; Wschebor, M. Almost sure oscillation of 
certain random processes, \emph{Bernoulli} \textbf{2} no. 3, 257--270 (1996). 
 
\bibitem {ber2}Bertoin, J.: \emph{Sur une int\'{e}grale pour les processus 
\`{a} {$\alpha$}-variation born\'{e}e.} Ann. Probab., \textbf{17}% 
(4):1521--1535, (1989). 
 
\bibitem {biagini}Biagini, F., Hu, Y., {\O }ksendal, B. and Zhang, T.: 
\emph{Stochastic calculus with respect to fractional Brownian motion and 
applications.} Probability and its Applications. Springer-Verlag, 2008. 
 
\bibitem {Borell}Borell, C. On polynomial chaos and integrability, 
\emph{Probability and Mathematical Statistics}, \textbf{3}, no. 2, 191--203 (1984). 
 
\bibitem {BM}Breuer, P., Major, P.: \emph{Central limit theorems for 
non-linear functionals of Gaussian fields}. J. Multivariate Anal. \textbf{13}, 
425--441 (1983). 
 
\bibitem {BS}Burdzy, K., Swanson, J.: \emph{A change of variable formula with 
It\^{o} correction term}. Ann. Probab. \textbf{38}, no. 5, 1817--1869 (2010). 
%Preprint, available at http://arxiv.org/PS\_cache/arxiv/pdf/0802/0802.3356v3.pdf. 
 
 
\bibitem {bruneau}Bruneau, M.: \emph{Variation totale d'une fonction.} Lecture 
Notes in Mathematics, Vol. \textbf{413}. Springer-Verlag, Berlin-New York, 1974. 
 
\bibitem {carmona}Carmona, Ph, Coutin, L. and Montseny, G.: \emph{Stochastic 
integration with respect to fractional Brownian motion.} Annales de l'Institut 
Henry Poincar\'e. Probabilit\'es et Statistiques, \textbf{39}, no 1, 27--68 (2003). 
 
\bibitem {dudley-norvaisa}Dudley, R. M., Norvai\v{s}a, R.: 
\emph{Differentiability of six operators on nonsmooth functions and }% 
$\emph{p}$\emph{-variation. } With the collaboration of Jinghua Qian. Lecture 
Notes in Mathematics, $\mathbf{1703}$. Springer-Verlag, Berlin, 1999. 
 
\bibitem {errami-russoCRAS}Errami, M., Russo, F.: \emph{ Covariation de 
convolution de martingales.} C. R. Acad. Sci. Paris S\'{e}r. I Math. 
\textbf{326}, no. 5, 601--606 (1998). 
 
\bibitem {RE}Errami, M., Russo, F.: \emph{$n$-covariation,generalized 
Dirichlet processes and calculus with respect to finite cubic variation 
processes.} Stochastic Processes Appl. , \textbf{104}, 259--299 (2003). 
 
\bibitem {flandoli-russo00}Flandoli, F., Russo, F.: {\emph{G}eneralized 
stochastic integration and stochastic ODE's.} Annals of Probability. Vol. 
\textbf{30}, no 1, 270--292 (2002). 
 
\bibitem {foellmer}F{\"{o}}llmer, H.: C{alcul d'{I}t\^{o} sans 
probabilit\'{e}s.} In \emph{S\'{e}minaire de Probabilit\'{e}s, XV (Univ. 
Strasbourg, Strasbourg, 1979/1980) (French)}, volume \textbf{850} of 
\emph{Lecture Notes in Math.}, pages 143--150. Springer, Berlin, 1981. 
 
%\bibitem {RE}Errami, M., Russo, F.: \emph{$n$-covariation,generalized 
%Dirichlet processes and calculus with respect to finite cubic variation 
%processes.} Stochastic Processes Appl. , \textbf{104}, 259-299 (2003). 
 
 
\bibitem {FV}Friz, P.; Victoir, N. \emph{Multidimensional Stochastic Processes 
as Rough Paths: Theory and Applications}. Cambridge Studies in Advanced 
Mathematics, Cambridge UP, 2010. 
 
\bibitem {GN2003}Gradinaru, M., Nourdin, I.: \emph{Approximation at first and 
second order of $m$-order integrals of the fractional Brownian motion and of 
certain semimartingales.} Electron. J. Probab. \textbf{8}, no. 18, 26 pp (2003). 
 
\bibitem {GRV}Gradinaru, M., Russo F., Vallois, P.: \emph{Generalized 
covariations, local time and {S}tratonovich {I}t\^{o}'s formula for fractional 
{B}rownian motion with {H}urst index {$H\geq\frac{1}{4}$}.} Ann. Pobab., 
\textbf{31} (4), 1772--1820, (2003). 
 
\bibitem {GNRV}Gradinaru, M., Nourdin, I., Russo, F., Vallois, P.: 
\emph{$m$-order integrals and It\^{o}'s formula for non-semimartingale 
processes; the case of a fractional Brownian motion with any Hurst index.} 
Ann. Inst. H. Poincar\'{e} Probab. Statist. \textbf{41}, 781--806 (2005). 
 
\bibitem {HNS}Hu, Y., Nualart,D., Song, J.: \emph{Fractional martingales and 
characterization of the fractional Brownian motion}. Ann. Probab. \textbf{37}, 
no. 6, 2404--2430 (2009). 
%Preprint, http://arxiv.org/abs/0711.1313. 
 
 
\bibitem {kruk}Kruk, I and Russo, F.: \emph{Malliavin-Skorohod calculus and 
Paley-Wiener integral for covariance singular processes}. Preprint HAL-INRIA 00540914. 
 
\bibitem {HHK}Kuo, H.-H. \emph{Introduction to stochastic integration}. 
Springer, 2006. 
 
\bibitem {lyons-qian}Lyons, T., Qian, Z.: \emph{System control and rough 
paths}. Oxford Mathematical Monographs. Oxford University Press, Oxford, 2002. 
 
\bibitem {MV}Mocioalca, O., Viens, F.: \emph{Skorohod integration and 
stochastic calculus beyond the fractional Brownian scale.} Journal of 
Functional Analysis, \textbf{222} no. 2, 385--434 (2004). 
 
\bibitem {N2}Nourdin, I.: \emph{A change of variable formula for the 2D 
fractional Brownian motion of Hurst index bigger or equal to $1/4$.} J. Funct. 
Anal. \textbf{256}, 2304--2320 (2009). 
 
\bibitem {NN}Nourdin, I., Nualart, D.: \emph{Central limit theorems for 
multiple Skorohod integrals}. J. Theoret. Probab. \textbf{23}, no. 1, 39--64. 
 
\bibitem {NNT}Nourdin, I., Nualart D., Tudor, C.: \emph{Central and 
non-central limit theorems for weighted power variations of fractional 
Brownian motion}. Ann. Inst. Henri Poincar\'{e} Probab. Stat. \textbf{46}, no. 
4, 1055--1079 (2010). 
%Preprint, available at http://arxiv.org/PS\_cache/arxiv/pdf/0710/0710.5639v1.pdf. 
 
 
\bibitem {NRS}Nourdin, I., R\'{e}veillac, A., Swanson, J.: \emph{The weak 
Stratonovich integral with respect to fractional Brownian motion with Hurst 
parameter 1/6}. Electron. J. Probab. \textbf{15}, no. 70, 2117--2162 (2010). 
 
\bibitem {Nbook}Nualart, D.: \emph{The Malliavin calculus and related topics.} 
Second edition. Probability and its Applications. Springer-Verlag, 2006. 
 
\bibitem {NOT}Nualart, D., Ortiz-Latorre, S.: \emph{ Central limit theorems 
for multiple stochastic integrals and Malliavin calculus.} Stochastic Process. 
Appl., \textbf{118}(4), 614--628 (2008). 
 
\bibitem {walsh}Rogers, L. C. G., Walsh, J. B.: \emph{ The exact 
$4/3$-variation of a process arising from Brownian motion.} Stochastics 
Stochastics Rep. \textbf{51}, no. 3-4, 267--291 (1994). 
 
\bibitem {TR}Russo, F., Tudor, C.: On the bifractional Brownian motion. 
\emph{Stochastic Processes Appl.} \textbf{116}, no. 6, 830--856 (2006). 
 
\bibitem {RV00}Russo, F., Vallois, P.: \emph{Stochastic calculus with respect 
to a finite quadratic variation process.} Stochastics and Stochastic Reports 
\textbf{70}, 1--40 (2000). 
 
\bibitem {RV}Russo, F., Vallois, P.: \emph{The generalized covariation process 
and It\^o formula.} Stochastic Process. Appl. \textbf{59}, no. 1, 81--104 (1995). 
 
\bibitem {RV2}Russo, F., Vallois, P.: \emph{Elements of stochastic calculus 
via regularizations.} S\'{e}minaire de Probabilit\'{e}s \textbf{XL}, Lect. 
Notes Math. \textbf{1899}, Berlin Heidelberg New-York, Springer, pages 
147--186 (2007). 
 
%\bibitem{swanson} Swanson, J.: 
%\emph{The calculus of differentials for the weak Stratonovich integral.} 
%Preprint 2011. 
 
 
\bibitem {TTV}Tindel, S., Tudor, C.A., Viens, F.: Sharp Gaussian regularity on 
the circle and application to the fractional stochastic heat equation. 
\emph{Journal of Functional Analysis}, \textbf{217} (2), 280--313 (2004). 
 
\bibitem {VV}Viens, F., Vizcarra, A.: \emph{Supremum Concentration Inequality 
and Modulus of Continuity for Sub-$n$th Chaos Processes.} J. Funct.Analysis 
\textbf{248}, 1--26 (2007). 
\end{thebibliography}
\end{document}